\documentclass[12pt,twoside]{article}
\usepackage{amssymb}
\usepackage{mathrsfs}
\usepackage{latexsym,amsfonts,amssymb,amsmath,amsthm}
\usepackage{graphicx}
\usepackage{amscd}
\usepackage{epstopdf}
\usepackage{color}
\usepackage{latexsym}
\usepackage{epsf}
\usepackage{multicol}
\usepackage{ifpdf}
\usepackage{lipsum,cuted}
\usepackage{float}
\usepackage{caption}
\usepackage{amsmath,amsfonts,amssymb,epsfig,subfigure}
\usepackage{mathrsfs}
\usepackage{stmaryrd}
\usepackage{graphicx,amsmath}
\usepackage{latexsym}
\usepackage{verbatim}
\usepackage{listings}

\renewcommand{\bar}{\overline}
\renewcommand{\hat}{\widehat}
\renewcommand{\tilde}{\widetilde}

\newtheorem{theorem}{Theorem}[section]
\newtheorem{assp}{Assumption}
\newtheorem{lemma}[theorem]{Lemma}
\newtheorem{remark}[theorem]{Remark}

\newtheorem{defn}{Definition}[section]

\definecolor{wco}{rgb}{0.5,0.2,0.3}

\makeatletter \@addtoreset{equation}{section}

\allowdisplaybreaks

\newcommand{\RR}{\mathbb{R}}

\def\nn{\nonumber}

\allowdisplaybreaks

\title{{\bf Explicit Numerical Approximations for McKean-Vlasov Neutral  Stochastic Differential Delay Equations}
}
\author{
{\bf Yuanping Cui$^{1}$, Xiaoyue Li$^{1}$,~~ Yi Liu$^{1, 2 }$, ~~Chenggui Yuan$^{3 }$}\\
 ~\\
\footnotesize{1. School of Mathematics and Statistics,
Northeast Normal University, Changchun,  130024, China.}\\
\footnotesize{2.  Department of Mathematics and Statistics, Auburn University, Auburn AL 36849, USA.}\\
\footnotesize{3.  Department of Mathematics, Swansea University, Bay Campus, SA1 8EN, UK.}\\
}

\begin{document}

\def\F{{\cal F}}
\def\diag{\hbox{\rm diag}}
\def\trace{\hbox{\rm trace}}\def\refer{\hangindent=0.3in\hangafter=1}
 \def\f{\varphi} \def\r{\rho}\def\e{\varepsilon}
\def\lf{\left} \def\rt{\right}\def\t{\triangle} \def\ra{\rightarrow}
\def\la{\label}\def\be{\begin{equation}}  \def\ee{\end{equation}}
\def\SS{{\mathbb{S}}}

\maketitle
\begin{abstract}
This paper studies the numerical methods to approximate the solutions for a sort of McKean-Vlasov neutral stochastic differential delay equations (MV-NSDDEs) that the growth of the drift coefficients is super-linear. First, We obtain that the solution of MV-NSDDE exists and is unique. Then, we use a stochastic particle method, which is on the basis of the results about the propagation of chaos between particle system and the original MV-NSDDE, to deal with the approximation of the law. Furthermore, we construct the tamed Euler-Maruyama numerical scheme with respect to the corresponding particle system and obtain the rate of convergence. Combining propagation of chaos and the convergence rate of the numerical solution to the particle system, we get a convergence error between the numerical solution and exact solution of the original MV-NSDDE in the stepsize and number of particles.

\end{abstract}



\section{Introduction}\label{s-w}
Due to the coefficients depending on the current distribution, McKean-Vlasov (MV) stochastic differential equations (SDEs) are also konwn as  distribution dependent (DD) SDEs described by
\begin{align*}
\mathrm{d}S(t)=b(t,S(t),\mathcal{L}_{t}^{S})\mathrm{d}t+\sigma(t,S(t),\mathcal{L}_{t}^{S})\mathrm{d}B(t),
\end{align*}
where~$\mathcal{L}_{t}^{S}$~is the law (or distribution) of~$S(t)$. The pioneering work of MV-SDEs has been done by McKean in \cite{ MHP1966, MHP1967, MHP1975} connected with a mathematical foundation of the Boltzmann equation. Under Lipschitz type conditions, it is known that MV-SDEs can be regarded as the result generated by the following interacting particle system
\begin{align*}
\mathrm{d}S^a(t)=b\bigg(t,S^a(t),\frac{1}{\Xi}\sum^{\Xi}_{j=1}\delta_{S^j(t)}\bigg)\mathrm{d}t
+\sigma\bigg(t,S^a(t),\frac{1}{\Xi}\sum^{\Xi}_{j=1}\delta_{S^j(t)}\bigg)\mathrm{d}B^{a}(t),~~~a=1,\cdots,\Xi,
\end{align*}
where~$\delta_{S^j(t)}$~denotes the Dirac measure at point~$S^j(t)$, as the particle number~$\Xi$~tends to infinity \cite{MHP1966, MHP1967}. Such results are also called propagation of chaos (see, e.g., \cite{LD2018}). Now, MV-SDEs  have been used as mathematical models widely in biological systems, financial engineering and  physics.

The theory of MV-SDEs has been developed rapidly, including  ergodicity \cite{EA2019}, Feyman-Kac Formula \cite{CD2018},
Harnack inequalities \cite{WFY2018} and so on. Recently, the existence of the unique solution for MV-SDEs which drift coefficients are super-linear growth attracts much attention \cite{DRG2019,WFY2018}. Due to the unsolvability of MV-SDEs, it is necessary to simulate them by some appropriate numerical methods. The explicit Euler-Maruyama (EM) method is common for approximating the solutions of SDEs with global Lipschitz coefficients \cite{KPE1992, MGN1995}. However, Hutzenthaler et al. \cite{H2011} gave several counter examples that the moments of order $p$ for EM numerical solutions diverge to infinity as the coefficients of the SDE are super-linear growth. Since explicit numerical methods are easily implementable and save the cost of computation, different approaches have been explored to modify the EM methods for the nonlinear SDEs (see, e.g., \cite{HM2012, LXY2019, SS2016}),
particularly, \cite{HM2012}  proposed  a tamed EM approximation for SDEs whose coefficients are locally Lipschitz continuous.  

The lag phenomena always happen in the real world. Stochastic differential delay equations (SDDEs) and stochastic differential delay equations with the neutral term (NSDDEs) are used commonly in many fields to represent a broad variety of natural and artificial systems (see, e.g., \cite{BBS2010}). In this article, we are interested in the approximation of MV-NSDDEs with drift coefficient satisfying super-linear growth condition. Compared with the approximation of NSDDEs,  the numerical method for MV-NSDDEs needs to approximate the law at each grid. Recently, a few works paid attention to the numerical methods
for MV-SDEs (see, e.g., \cite{AF2002,  JX2019, BK20, TD1996, BM1997}). Especially, using the tamed EM scheme and the theory of propagation of chaos, Dos-Reis et al. \cite{B21, GSG2018} gave the strong convergence of the approximation solutions for MV-SDEs, \cite{RW19} studied least squares estimator of a kind of path-dependent MV-SDEs by  adopting  a tamed EM algorithm. However, as well as we know, there are few results on the numerical approximations of MV-NSDDEs. Therefore, the main goal of our paper is to establish the strong convergence theory for MV-NSDDEs using the tamed EM method. The neutral-type delay and the law are the key obstacles of the numerical approximation.

According to the ideas from \cite{B21,GSG2018, RW19}, we develop the tamed EM scheme to the particle system converging to the MV-NSDDEs, show the strong convergence between the numerical solutions and the exact solutions of the particle system and further estimate the convergence rate. Finally, using the result of propagation of chaos, we obtain the  strong convergence for MV-NSDDEs and its rate.

The rest organization of this paper is as follows:
In section 2, we introduce some notations and preliminaries. In section 3, 
we give the existence of the unique solution for the MV-NSDDE and introduce the result on propagation of chaos of the particle system. In section 4, we construct the tamed EM scheme for the particle system and obtain the moment boundedness for the corresponding approximation solutions. Furthermore, we estimate the strong convergence rate between the numerical solutions and exact solutions of the particle system, and then get the strong convergence for MV-NSDDEs. In section 5, the numerical simulation of an example is provided to demonstrate the validity of our numerical algorithm.

\section{Notations and Preliminaries}\label{n-p}
 Assume $( \Omega,~\mathcal{F},~\mathbb{P} )$~is a probability space that is complete and has a normal filtration~$\{\mathcal{F}_{t}\}_{t\geq0}$~satisfying the usual conditions (namely, it is right continuous and increasing while~${\mathcal{F}}_{0}$~contains all $\mathbb{P}$-null sets). Let $\mathbb{N}=\{1,2,\cdots\}$~and~$d,~m\in\mathbb{N}$. Denote~$\{B(t)\}_{t\geq0}$~as a standard~$m$-dimensional Brownian motion on the probability space. Let~$\RR_+=\{x\in\RR: x\geq0\}$,~$\RR^{d}$~be~$d$-dimensional Euclidean space, and $\RR^{d\times m}$~be the space of real~$d\times m$-matrices. If~$x\in\RR^{d}$, then denote by~$|x|$~the Euclidean norm. For any matrix~$A$, we denote  its transpose by $A^{T}$ and define its trace norm by $|A|=\sqrt{\mathrm{trace}(AA^{T})}$. Moreover, for any $x\in \RR$~and~$y\in\RR$, we use the notation~$x\wedge y=\min\{x,y\}$~and~$x\vee y=\max\{x,y\}$.
 Let $\emptyset$~denote the empty set and $\inf\emptyset=\infty$. For any~$x\in\RR$, we use $\lfloor x\rfloor$~as the integer part of~$x$.

 For any $q>0$, let~$L^{q}=L^{q}(\Omega;\RR^{d})$ be the space of~$\RR^{d}$-valued random variables $Z$ satisfying $\mathbb{E}[|Z(\omega)|^q]<+\infty$. We also let $\mathcal{L}^{Z}$ be the probability law (or distribution) of a random variable~$Z$. Let $\delta_{x}(\cdot)$~denote the Dirac delta measure concentrated at a point~$x\in\RR^{d}$ and $\mathcal{P}(\RR^{d})$~be a collection of probability measures on~$\RR^{d}$. For $q\geq1$, we let $\mathcal{P}_q(\RR^{d})$ be a collection of probability measures on~$\RR^{d}$ with finite $q$th moments, and define \begin{align*}
W_q(\mu):=\big(\int_{\RR^{d}}|x|^q\mu(\mathrm{d}x)\big)^{\frac{1}q},~ ~~~ \forall  \mu\in\mathcal{P}_q(\RR^{d}).
\end{align*}

\begin{lemma}\label{defn2.1} {\rm\cite[pp.106-107]{VC2008} (  Wasserstein~Distance
)}
Let~$q\geq1$. Define
$$
\mathbb{W}_q(\mu,\nu):=\inf_{\pi \in \mathcal{D}(\mu,\nu)}\bigg\{\int_{\RR^{d}}|x-y|^q\pi(\mathrm{d}x,\mathrm{d}y)\bigg\}^{\frac{1}q},
~\mu,\nu\in\mathcal{P}_q(\RR^{d}),
$$
where~$\mathcal{D}(\mu,\nu)$~denotes the family of all couplings for~$\mu$~and~$\nu$. Then~$\mathbb{W}_q$~is a distance on~$\mathcal{P}_q(\RR^{d})$.
\end{lemma}

\section{ MV-NSDDEs}\label{s-c}
Throughout this paper, let~$\tau>0$~and~$\mathcal{C}([-\tau,0];\mathbb{R}^{d})$~be the space of continuous functions~$\phi$ from~$[-\tau,0]$~to~$\mathbb{R}^{d}$~and~$\|\phi\|=\sup\limits_{-\tau\leq\theta\leq0}|\phi(\theta)|$. For any~$q\geq 0$, let $L^q_{\mathcal{F}_{0}}([-\tau,0];\mathbb{R}^{d})$ denote the set of~$\mathcal{F}_{0}$-measurable~$\mathcal{C}([-\tau,0];\mathbb{R}^{d})$-valued random variables $\xi$ with $\mathbb{E}[\|\xi\|^q]<\infty$.
For any given $T\in(0,\infty)$, consider the MV-NSDDE
\begin{equation} \label{eq3.1}
\mathrm{d}\Big(S(t)-D(S(t-\tau))\Big)=b(S(t),S(t-\tau),\mathcal{L}_{t}^{S})\mathrm{d}t+\sigma(S(t),S(t-\tau),\mathcal{L}_{t}^{S})\mathrm{d}B(t),~~~t\in[0,T],
\end{equation}
where $\mathcal{L}_{t}^{S}$ is the law of~$S(t)$.
$
D:\RR^{d}\rightarrow\RR^{d},~b:\RR^{d}\times\RR^{d}\times\mathcal{P}_{2}(\RR^{d})\rightarrow\RR^{d},~\sigma:\RR^{d}\times\RR^{d}\times\mathcal{P}_{2}(\RR^{d})\rightarrow\RR^{d\times m}
$
are Borel-measurable, and~$b,~\sigma$~are continuous on~$\RR^{d}\times\RR^{d}\times\mathcal{P}_{2}(\RR^{d})$.
 Define $S_t(\theta)=S(t+\theta)$, $-\tau\leq \theta\leq 0$. Then $\{S_t\}_{t\in[0,T]}$ is thought of as a $\mathcal{C}([-\tau,0];\mathbb{R}^{d})$-valued stochastic process.
Assume that the initial condition satisfies
 \begin{equation} \label{eq3.2}
 S_{0} =\xi\in L^{p}_{\mathcal{F}_{0}}([-\tau,0];\mathbb{R}^{d}),~~\hbox{and}~~ \mathbb{E}\big[\sup_{_{-\tau\leq r,t\leq0}}|\xi(t)-\xi(r)|^{p}\big]\leq K_0|t-r|^{\frac{p}{2}}
\end{equation}   for some $p\geq 2$ and $K_0>0$.
And for the sake of simplicity, unless otherwise stated,~$C$~denotes the positive constant whose value may vary with different places of this paper.
\begin{defn}\label{dfn3.1}$\mathrm{(MV-NSDDE~strong~solution~and~uniqueness)}$~~An~$\RR^{d}$-valued stochastic process~$\{S(t)\}_{t\in[-\tau,T]}$~is called a strong solution to~$(\ref{eq3.1})$~if it satisfies the following conditions:
\begin{itemize}
\item[$(1)$]  it is continuous and~$\{S(t)\}_{t\in[0,T]}$~is~$\{\mathcal{F}_{t}\}$-adapted;
\item[$(2)$]  $\mathbb{E}\Big[\int_{0}^{T} |b(S(t),S(t-\tau),\mathcal{L}_{t}^{S})|\mathrm{d}t\Big]<+\infty,~\mathbb{E}\Big[\int_{0}^{T}|\sigma(S(t),S(t-\tau),\mathcal{L}_{t}^{S})|^{2}\mathrm{d}t\Big]<+\infty$;
\item[$(3)$]  $S_0=\xi$,~and for any $t\in[0, T]$,
\begin{align*}
S(t)-D(S(t-\tau))=&\xi(0)-D(\xi(-\tau))+\int_{0}^{t}b(S(r),S(r-\tau),\mathcal{L}_{r}^{S})\mathrm{d}r \\
&+\int_{0}^{t}\sigma(S(r),S(r-\tau),\mathcal{L}_{r}^{S})\mathrm{d}B(r)~~~ ~ \hbox{a.s}.
\end{align*}
\end{itemize}
We say that the solution~$\{S(t)\}_{t\in[-\tau,T]}$~is unique if for any other solution~$\{\bar{S}(t)\}_{t\in[-\tau,T]}$,
$$
P\{S(t)=\bar{S}(t)~for~all~-\tau\leq t\leq T\}=1,
$$
\end{defn}

To prove the main  results of this paper, we prepare several lemmas.
\begin{lemma}\label{l3.1}{\rm\cite{DRG2019} }~~For any~$\mu\in\mathcal{P}_{2}(\RR^{d})$, $\mathbb{W}_{2}(\mu,\delta_{0})=W_{2}(\mu)$.
\end{lemma}
\begin{lemma}\label{l3.2}{\rm\cite[pp.362, Theorem 5.8]{CR2018}}~~For~$\{Z_{n}\}_{n\geq1}$~a sequence of independent identically distributed (i.i.d. for short) random variables in~$\RR^{d}$~with common distribution~$\mu\in\mathcal{P}(\RR^{d})$~and a constant~$ \Xi\in\mathbb{N}$, we define the empirical measure
$
\bar{\mu}^{\Xi}:=\frac{1}{\Xi}\sum\limits^{\Xi}_{j=1}\delta_{Z_{j}}.
$
If~$\mu\in\mathcal{P}_{q}(\RR^{d})$~with~$q>4$, then
there exists a constant~$C=C(d,q,W_{q}(\mu))$~such that  for each~$\Xi\geq2$,
\begin{align*}
\mathbb{E}\big[\mathbb{W}^{2}_{2}(\bar{\mu}^{\Xi},\mu)\big]\leq C
\left\{
\begin{array}{lll}
\Xi^{-1/2},&~~~1\leq d<4,\\
\Xi^{-1/2}\log(\Xi),&~~~d=4,\\
\Xi^{-2/d},&~~~ 4<d.
\end{array}
\right.
\end{align*}
\end{lemma}

For the theoretical results of this paper including the existence of a unique solution and convergence of numerical approximation for~$(\ref{eq3.1})$-$(\ref{eq3.2})$, we impose the following assumptions.

\begin{assp}\label{a1} $D(0)=0$. And there exist positive constants~ $\lambda<1$ and $K_i~ (i=1,2,3)$ such that
\begin{align*}
  &|D(x_{_{1}})-D(x_{_{2}})|\leq\lambda|x_{1}-x_{2}|,~ ~ |b(0,0,\mu)|\vee|\sigma(0,0,\mu)|\leq K_1\big(1+W_{2}(\mu)\big),\\
&(x_{1}-D(y_{1})-x_{2}+D(y_{2}))^{T} \big(b(x_{1},y_{1},\mu_{1})-b(x_{2},y_{2},\mu_{2})\big)\\
&~~~~~~~~~~~~~~~~~~~~~~~~~~~~~~~~~~\leq K_2\big(|x_{1}-x_{2}|^{2}+|y_{1}-y_{2}|^{2}+\mathbb{W}^{2}_{2}(\mu_{1},\mu_{2})\big)\\
and\\
&|\sigma(x_{1},y_{1},\mu_{1})-\sigma(x_{2},y_{2},\mu_{2})|^2
\leq K_3\big(|x_{1}-x_{2}|^{2}+|y_{1}-y_{2}|^{2}+\mathbb{W}^{2}_{2}(\mu_{1},\mu_{2})\big)
\end{align*}
hold for any $x_{1},~x_{2},~y_{1},~y_{2}\in\RR^{d}$,~$\mu,~\mu_{1},~\mu_{2}\in\mathcal{P}_{2}(\RR^{d})$.
 \end{assp}
\begin{assp}\label{a2}
There exist positive constants $K_4 $~and~$c$~such that
$$
|b(x_{1},y_{1},\mu)-b(x_{2},y_{2},\mu)|\leq K_4\big(1+|x_{1}|^{c}+|y_{1}|^{c}+|x_{2}|^{c}+|y_{2}|^{c}\big)\big(|x_{1}-x_{2}|^{2}+|y_{1}-y_{2}|^{2}\big)
$$
holds for any~$x_{1},~x_{2},~y_{1},~y_{2}\in\RR^{d}$, and~$\mu\in\mathcal{P}_{2}(\RR^{d})$.
\end{assp}

\begin{remark}\label{r3.3}
Under  Assumption~$\ref{a1}$, by virtue of the elementary inequality and Lemma~$\ref{l3.1}$, it is easy to show that
\begin{align} \label{eq3.3}
|\sigma(x,y,\mu)|^2&\leq 2\big|\sigma(x,y,\mu)-\sigma(0,0,\delta_{0})\big|^2+2|\sigma(0,0,\delta_{0})|^2\nn\\
&\leq 2K_3\big(|x|^2+|y|^2+\mathbb{W}_{2}^2(\mu,\delta_{0})\big)+2K^{2}_1(1+W_{2}(\delta_{0}))^2\nn\\
&\leq 2 (K_3+K^{2}_1)\big(1+|x|^2+|y|^2+W_{2}^2(\mu)\big).
\end{align}
 Additionally, one obtains
\begin{align} \label{eq3.4}
|D(x)|=|D(x)-D(0)|\leq\lambda|x|.
\end{align}
This,  together with~$\mathrm{Young}$'s~inequality,  implies that
\begin{align}\label{eq3.5}
\big(x-D(y)\big)^{T}b(x,y,\mu)&=(x-D(y)-0+D(0))^{T}\big(b(x,y,\mu) -b(0,0,\delta_{0})+b(0,0,\delta_{0})\big)\nn\\
&\leq K_2 \big( |x|^{2}+|y|^{2}+W^{2}_{2}(\mu) \big)+ (|x|+|D(y)|)|b(0,0,\delta_{0})|\nn\\
&\leq(K_1+  K_2) \big( 1+|x|^{2}+|y|^{2}+W^{2}_{2}(\mu) \big).
\end{align}
Moreover, Assumptions \ref{a1} and \ref{a2} also yield that
\begin{align}\label{eq3.6}
|b(x,y,\mu)|&\leq|b(x,y,\mu)-b(0,0,\mu)|+|b(0,0,\mu) | \nn\\
&\leq K_{4}\big(1+|x|^{c}+|y|^{c}\big)\big(|x|+|y|\big)+K_{1}(1+W_{2}(\mu)) \nn\\
&\leq4(K_{1}+K_{4})\big(1+|x|^{c+1}+|y|^{c+1}+W_{2}(\mu)\big).
\end{align}
\end{remark}

\begin{lemma}\label{l3.4}{\rm \cite[pp.211, Lemma 4.1]{MX2007}}~~Let~$q>1,~\varepsilon>0$~and~$x,y\in\mathbb{R}$. Then
$$
|x+y|^{q}\leq(1+\varepsilon^{\frac{1}{q-1}})^{q-1}(|x|^{q}+\frac{|y|^{q}}{\varepsilon}).
$$
Especially for  $q=p, ~\varepsilon=( {(1-\lambda)}/{\lambda})^{p-1}$, where $p\geq 2$ and $0<\lambda<1$ are given in \eqref{eq3.2} and Assumption \ref{a1}, respectively, one can derive the following inequality
\begin{align}\label{eq3.7}
|x+y|^{p}\leq \frac{1}{\lambda ^{p-1}} |x|^{p}+\frac{1}{(1-\lambda )^{p-1}} |y|^{p} .
\end{align}
\end{lemma}

\subsection{Existence and Uniqueness of Solutions}\label{s3.1}
Now we begin to investigate the existence and uniqueness of the exact solution to~$(\ref{eq3.1})$-$(\ref{eq3.2})$.
\begin{theorem}\label{th3.5}
Under  Assumption~$\ref{a1}$, there is a unique strong solution to the equation~$(\ref{eq3.1})$-$(\ref{eq3.2})$. Moreover, for any $T>0$,
\begin{align}\label{eq3.8}
\mathbb{E}\big[\sup_{0\leq t\leq T}|S(t)|^{p}\big]\leq V(1+T)e^{V(1+T)T},
\end{align}
where~$V: =L(p,K_{1},K_{2},K_{3},\mathbb{E}[\|\xi\|^{p}])$ depends on $p,K_{1},K_{2},K_{3},\mathbb{E}[\|\xi\|^{p}]$.
\end{theorem}
\begin{proof}
We borrow the techniques for MV-SDEs from \cite{WFY2018}~and \cite{HY2020}. Then we divide the proof which is rather technical  into 3 steps.\\
\underline{Step 1.}
For~$t\in[0,T]$, we define
$$
 S^{(0)}(\theta)=\xi(\theta),~~~\theta\in[-\tau,0];~~~~S^{(0)}(t)\equiv\xi(0).
$$
Then one observes that~$\mathcal{L}_{t}^{S^{(0)}}\equiv\mathcal{L}^{\xi(0)}$.
For~$i\geq1$, let~$S^{(i)}(t) $~solve NSDDE
\begin{align} \label{eq3.9}
 \mathrm{d}\Big(S^{(i)}(t)-D(S^{(i)}(t-\tau))\Big)=& b(S^{(i)}(t),S^{(i)}(t-\tau),\mathcal{L}_{t}^{S^{(i-1)}})\mathrm{d}t\nn\\
&+\sigma(S^{(i)}(t),S^{(i)}(t-\tau),\mathcal{L}_{t}^{S^{(i-1)}})\mathrm{d}B(t),~~~~t\in[0,T],
\end{align}
and the initial data is given by
\begin{align} \label{eq3.10}
 S^{(i)} (\theta)=\xi(\theta),~~~\theta\in[-\tau,0],
\end{align}
where~$\mathcal{L}_{t}^{S^{(i-1)}}$~denotes the law of~$ S^{(i-1)}(t)$.  Now, we claim that~$(\ref{eq3.9})$-$(\ref{eq3.10})$~has a unique solution and the solution has the property that
$$
\mathbb{E}\big[\sup\limits_{0\leq t\leq T}|S^{(i)}(t)|^{p}\big]<\infty.
$$
As $i=1$, it follows by virtue of~\cite{JSY2017}~that under Assumption~$\ref{a1}$, NSDDE~$(\ref{eq3.9})$-$(\ref{eq3.10})$~exists the unique regular solution~$S^{(1)}(t)$. Moreover, due to $(\ref{eq3.7})$ and Assumption \ref{a1}, one can know that
\begin{align*}
|S^{(1)}(t)|^{p}&=|D(S^{(1)}(t-\tau))+S^{(1)}(t)-D(S^{(1)}(t-\tau))|^{p}\nn\\
&\leq\frac{1}{ \lambda^{p-1}}|D(S^{(1)}(t-\tau))|^{p} +\frac{1}{(1-\lambda)^{p-1}}|S^{(1)}(t)-D(S^{(1)}(t-\tau))|^{p}\nn\\
&\leq
 \lambda| S^{(1)}(t-\tau) |^{p} +\frac{1}{(1-\lambda)^{p-1}}|S^{(1)}(t)-D(S^{(1)}(t-\tau))|^{p}
\end{align*}
for any $t\in[0,T]$. This implies that
\begin{align*}
\sup_{0\leq r\leq t}|S^{(1)}(r)|^{p}&\leq
\lambda\|\xi\|^{p}+\lambda\sup_{0\leq r\leq t}|S^{(1)}(r)|^{p}+\frac{1}{(1-\lambda)^{p-1}}\sup_{0\leq r\leq t}\big|S^{(1)}(r)-D(S^{(1)}(r-\tau))\big|^{p}.
\end{align*}
Therefore, we have
\begin{align} \label{eq3.12}
\sup_{0\leq r\leq t}|S^{(1)}(r)|^{p}
&\leq
\frac{\lambda}{1-\lambda}\|\xi\|^{p}+\frac{1}{(1-\lambda)^{p}}\sup_{0\leq r\leq t}\big|S^{(1)}(r)-D(S^{(1)}(r-\tau))\big|^{p}.
\end{align}
The~It$\hat{\mathrm{o}}$~formula leads to
\begin{align*}
|S^{(1)}(t)-D(S^{(1)}(t-\tau))|^{p}\leq |\xi(0)-D(\xi(-\tau))|^{p}+J_{1}(t)+J_{2}(t)+J_{3}(t),
\end{align*}
where
\begin{align*}
J_{1}(t)&=p\int_{0}^{t}\big|S^{(1)}(r)-D(S^{(1)}(r-\tau))\big|^{p-2}\\
&~~~~~~~~\times\Big(S^{(1)}(r)-D(S^{(1)}(r-\tau))\Big)^{T}b(S^{(1)}(r),S^{(1)}(r-\tau),\mathcal{L}_{r}^{S^{(0)}})\mathrm{d}r,
\end{align*}
\begin{align*}
J_{2}(t)&=\frac{p(p-1)}{2}\int_{0}^{t}\big|S^{(1)}(r)-D(S^{(1)}(r-\tau))\big|^{p-2}\big|\sigma(S^{(1)}(r),S^{(1)}(r-\tau),\mathcal{L}_{r}^{S^{(0)}})\big|^{2}\mathrm{d}r
\end{align*}
and
\begin{align*}
J_{3}(t)&=p\int_{0}^{t}\big|S^{(1)}(r)-D(S^{(1)}(r-\tau))\big|^{p-2}\\
&~~~~~~~~\times\big(S^{(1)}(r)-D(S^{(1)}(r-\tau))\big)^{T}\sigma(S^{(1)}(r),S^{(1)}(r-\tau),\mathcal{L}_{r}^{S^{(0)}})\mathrm{d}B(r).
\end{align*}
For any positive integer~$N$, we define the stopping time
$$
\varsigma^{(1)}_{N}=T\wedge\inf\big\{t\in[0,T]:~|S^{(1)}(t)|\geq N\big\}.
$$
Obviously, $\varsigma^{(1)}_{N}\uparrow T$~a.s.~as~$N\rightarrow\infty$.
Then, one can write that
\begin{align} \label{eq3.14}
&\mathbb{E}\Big[\sup_{0\leq r\leq t\wedge\varsigma^{(1)}_{N}}\big|S^{(1)}(r)-D(S^{(1)}(r-\tau))\big|^{p}\Big]\nn\\
&\leq \mathbb{E}\big[\big|\xi(0)-D(\xi(-\tau))\big|^{p}\big]+\mathbb{E}\big[\sup_{0\leq r\leq t\wedge\varsigma^{(1)}_{N}}\big(J_{1}(r)+J_{2}(r)\big)\big]+\mathbb{E}\big[\sup_{0\leq r\leq t\wedge\varsigma^{(1)}_{N}}J_{3}(r)\big].
\end{align}
By~$(\ref{eq3.3})$~and~$(\ref{eq3.5})$, we obtain that
\begin{align*}
&\mathbb{E}\Big[\sup_{0\leq r\leq t\wedge\varsigma^{(1)}_{N}}\big(J_{1}(r)+J_{2}(r)\big)\Big]\nn\\
&\leq C\mathbb{E}\Big[\int_{0}^{t\wedge\varsigma^{(1)}_{N}}\big|S^{(1)}(r)-D(S^{(1)}(r-\tau))\big|^{p-2}\nn\\
&~~~~~~~~~~~~~~~~~~~~~~~~~~\times\big(1+|S^{(1)}(r)|^{2}+|S^{(1)}(r-\tau)|^{2}+W_{2}^{2}(\mathcal{L}_{r}^{S^{(0)}})\big)\mathrm{d}r\Big]\nn\\
&\leq C\mathbb{E}\Big[\int_{0}^{t\wedge\varsigma^{(1)}_{N}}\big(|S^{(1)}(r)|^{p-2}+|D(S^{(1)}(r-\tau))|^{p-2}\big)\nn\\
&~~~~~~~~~~~~~~~~~~~~~~~~~~\times\big(1+|S^{(1)}(r)|^{2}+|S^{(1)}(r-\tau)|^{2}+\mathbb{E}[|\xi(0)|^2]\big)\mathrm{d}r\Big].
\end{align*}
Using Young's inequality and~$\mathrm{H\ddot{o}lder}$'s inequality, we derive that
\begin{align}\label{eq3.15}
\mathbb{E}\big[\sup_{0\leq r\leq t\wedge\varsigma^{(1)}_{N}}(J_{1}(r)+J_{2}(r))\big]&\leq C\mathbb{E}\Big[\int_{0}^{t\wedge\varsigma^{(1)}_{N}}
\big(1+|S^{(1)}(r)|^{p}+|S^{(1)}(r-\tau)|^{p}+\mathbb{E}[|\xi(0)|^p]\big)\mathrm{d}r\Big]\nn\\
&\leq C\Big(\mathbb{E}\Big[\int_{0}^{t\wedge\varsigma^{(1)}_{N}}\big(|S^{(1)}(r)|^{p}+|S^{(1)}(r-\tau)|^{p}
 \big)\mathrm{d}r\Big]+1\Big)\nn\\
&\leq C\Big(\mathbb{E}\Big[\int_{0}^{t\wedge\varsigma^{(1)}_{N}}|S^{(1)}(r)|^{p}\mathrm{d}r\Big]
 +\int_{-\tau}^{0}\mathbb{E}\big[\|\xi\|^{p}\big]\mathrm{d}r+1\Big) \nn\\
&\leq C\Big(\int_{0}^{t}\mathbb{E}\big[\sup_{0\leq v\leq r\wedge\varsigma^{(1)}_{N}}|S^{(1)}(v)|^{p}\big]\mathrm{d}r+1\Big).
\end{align}
The application of the Burkholder-Davis-Gundy (BDG) inequality, Young's inequality and~$\mathrm{H\ddot{o}lder}$'s inequality yields
\begin{align}\label{eq3.16}
&\mathbb{E}\big[\sup_{0\leq r\leq t\wedge\varsigma^{(1)}_{N}}J_{3}(r)\big]\nn\\
&\leq C\mathbb{E}\Big[\Big(\int_{0}^{t\wedge\varsigma^{(1)}_{N}}\big|S^{(1)}(r)-D(S^{(1)}(r-\tau))\big|^{2p-2}
\big|\sigma(S^{(1)}(r),S^{(1)}(r-\tau),\mathcal{L}_{r}^{S^{(0)}})\big|^{2}\mathrm{d}r\Big)^{\frac{1}{2}}\Big]\nn\\
&\leq C\mathbb{E}\Big[\sup_{0\leq r\leq t\wedge\varsigma^{(1)}_{N}}\big|S^{(1)}(r)\!-\!D(S^{(1)}(r-\tau))\big|^{p-1}
\big(\int_{0}^{t\wedge\varsigma^{(1)}_{N}}\big|\sigma(S^{(1)}(r),S^{(1)}(r\!-\!\tau),\mathcal{L}_{r}^{S^{(0)}})\big|^{2}\mathrm{d}r\big)^{\frac{1}{2}}\Big]\nn\\
&\leq\frac{1}{2}\mathbb{E}\Big[\sup_{0\leq r\leq t\wedge\varsigma^{(1)}_{N}}\big|S^{(1)}(r)-D(S^{(1)}(r-\tau))\big|^{p}\Big]\nn\\
&~~~~+C\mathbb{E}\Big[\Big(\int_{0}^{t\wedge\varsigma^{(1)}_{N}}\big(1+|S^{(1)}(r)|^{2}+|S^{(1)}(r-\tau)|^{2}+W^{2}_{2}(\mathcal{L}_{r}^{S^{(0)}})\big)\mathrm{d}r\Big)^{\frac{p}{2}}\Big]\nn\\
&\leq\frac{1}{2}\mathbb{E}\Big[\sup_{0\leq r\leq t\wedge\varsigma^{(1)}_{N}}\big|S^{(1)}(r)\!-\!D(S^{(1)}(r\!-\!\tau))\big|^{p}\!\Big]\!+\!C\Big(\int_{0}^{t}\mathbb{E}\Big[\sup_{0\leq v\leq r\wedge\varsigma^{(1)}_{N}}\!|S^{(1)}(v)|^{p}\Big]\mathrm{d}r\!+\!1\Big).
\end{align}
Inserting~$(\ref{eq3.15})$~and~$(\ref{eq3.16})$ into $(\ref{eq3.14})$ yields
\begin{align*}
\mathbb{E}\Big[\sup_{0\leq r\leq t\wedge\varsigma^{(1)}_{N}}\big|S^{(1)}(r)-D(S^{(1)}(r-\tau))\big|^{p}\Big]\leq C \int_{0}^{t}\mathbb{E}\Big[\sup_{0\leq v\leq r\wedge\varsigma^{(1)}_{N}}|S^{(1)}(v)|^{p}\Big]\mathrm{d}r+C ,
\end{align*}
which together with~$(\ref{eq3.12})$~and the Gronwall inequality  implies
\begin{align*}
\mathbb{E}\big[\sup_{0\leq r\leq t\wedge\varsigma^{(1)}_{N}}|S^{(1)}(r)|^{p}\big]\leq C.
\end{align*}
In view of Fatou's lemma, we can deduce that
\begin{align} \label{eq3.17}
\mathbb{E}\big[\sup_{0\leq r\leq t}|S^{(1)}(r)|^{p}\big]\leq C,   ~~~~~0\leq t\leq T,
\end{align}
 holds for~$i=1$.  Next, assume that the assertion
$$
\mathbb{E}[\sup\limits_{0\leq t\leq T}|S^{(i)}(t)|^{p}]<\infty
$$
holds for~$i=k$.
By changing~ $(S^{(1)},\mathcal{L}^{S^{(0)}},S^{(0)})$ to $(S^{(i+1)},\mathcal{L}^{S^{(i)}},S^{(i)})$ in the above proof and repeating that procedure, we can obtain that
 for~$i=k+1$ the assertion~still holds.

\underline{Step 2.}  We shall  prove that both the  sequence $\{S^{(i)}(\cdot)\}$  and its law $\{\mathcal{L}^{S^{(i)}}\}$ have limit. To this end,  we  show that
\begin{align}\label{eq3.18}
\mathbb{E}\big[\sup_{0\leq r\leq T_0}|S^{(i+1)}(r)-S^{(i)}(r)|^{p}\big]\leq Ce^{-i},~~~~\hbox{for some} ~ T_0\in (0, T]  .
\end{align}
Due to~$(\ref{eq3.7})$ and Assumption \ref{a1} one observes
\begin{align*}
|S^{(i+1)}(t)-S^{(i)}(t)|^{p}
&\leq  \lambda\big|S^{(i+1)}(t -\tau)-S^{(i)}(t-\tau)\big|^{p}
+\frac{1}{(1-\lambda)^{p-1}}|\Gamma^{(i)}(t)|^{p} ,
\end{align*}
where $\Gamma^{(i)}(t)= S^{(i+1)}(t)-D(S^{(i+1)}(t-\tau))-S^{(i)}(t)+D(S^{(i)}(t-\tau)) $,  which implies
\begin{align*}
&\sup_{0\leq r\leq t}\big|S^{(i+1)}(r)-S^{(i)}(r)\big|^{p}\nn\\
&\leq\lambda\sup_{-\tau\leq r\leq 0}\big|S^{(i+1)}(r)\!-\!S^{(i)}(r)\big|^{p}\!+\!\lambda\sup_{0\leq r\leq t}\big|S^{(i+1)}(r)\!-\!S^{(i)}(r)\big|^{p}
+\frac{1}{(1-\lambda)^{p-1}}\sup_{0\leq r\leq t}|\Gamma^{(i)}(r)|^{p}.
\end{align*}
Noting that $S^{(i+1)}(t)$ and $S^{(i)}(t)$ have the same initial value, we drive that
\begin{align}\label{eq3.19}
\sup_{0\leq r\leq t}\big|S^{(i+1)}(r)-S^{(i)}(r)\big|^{p}
\leq\frac{1}{(1-\lambda)^{p}}\sup_{0\leq r\leq t}|\Gamma^{(i)}(r)|^{p}.
\end{align}
By the application of the~It$\hat{\mathrm{o}}$~formula we have
\begin{align}\label{eq3.20}
\mathbb{E}\big[\sup_{0\leq s\leq t} |\Gamma^{(i)}(s)|^{p}\big]\leq  H_{1}+H_{2} ,
\end{align}
where
\begin{align*}
H_{1}&=p\mathbb{E}\bigg[\sup_{0\leq r\leq t}\int_{0}^{r}|\Gamma^{(i)}(v)|^{p-2}\\
&~~~~~~~~\times\Big\{\big(\Gamma^{(i)}(v)\big)^{T}\Big(b(S^{(i+1)}(v),S^{(i+1)}(v\!-\tau),\mathcal{L}_{v}^{S^{(i)}})\!-b(S^{(i)}(v),S^{(i)}(v\!-\tau),\mathcal{L}_{v}^{S^{(i-1)}})\Big)\\
&~~~~~~~~~~~~+\!\frac{p-1}{2}\Big|\sigma(S^{(i+1)}(v)\!,S^{(i+1)}(v\!-\!\tau),\!\mathcal{L}_{v}^{S^{(i)}})
 \!-\!\sigma(S^{(i)}(v),S^{(i)}(v\!-\!\tau),\mathcal{L}_{v}^{S^{(i-1)}})\Big|^{2}\Big\}\mathrm{d}v\bigg]\\
\end{align*}
 and
\begin{align*}
H_{2}&=p\mathbb{E}\bigg[\sup_{0\leq r\leq t}\int_{0}^{r}|\Gamma^{(i)}(v)|^{p-2}\big(\Gamma^{(i)}(v)\big)^{T}\\
&~~~~~~~~\times \Big(\sigma(S^{(i+1)}(v),S^{(i+1)}(v-\tau),\mathcal{L}_{v}^{S^{(i)}})
-\sigma(S^{(i)}(v),S^{(i)}(v-\tau),\mathcal{L}_{v}^{S^{(i-1)}})\Big)\mathrm{d}B(v)\bigg].
\end{align*}
It follows from Assumption~$\ref{a1}$, Young's inequality and~$\mathrm{H\ddot{o}lder}$'s~inequality that
\begin{align}\label{eq3.21}
H_{1}&\leq C\mathbb{E}\bigg[ \int_{0}^{t}|\Gamma^{(i)}(r)|^{p-2}\nn\\
&~~~~~~~~\times\Big(\big|\!S^{(i+1)}(r)\!-\!S^{(i)}(r)\big|^{2}\!+\!\big|S^{(i+1)}(r\!-\!\tau)\!-S^{(i)}(r\!-\!\tau)\big|^{2}\!+\!\mathbb{W}^{2}_{2}(\mathcal{L}_{r}^{S^{(i)}},\!\mathcal{L}_{r}^{S^{(i-1)}})\Big)\mathrm{d}r\!\bigg]\nn\\
&\leq C\mathbb{E}\bigg[\int_{0}^{t}\bigg\{\Big|\big(S^{(i+1)}(r)\!-\!S^{(i)}(r)\big)\!-\!\big(D(S^{(i+1)}(r\!-\!\tau))\!-\!D(S^{(i)}(r\!-\!\tau))\big)\Big|^{p} \nn\\
&~~~~~~~~+\!\Big(\!|S^{(i+1)}(r)\!-\!S^{(i)}(r)|^{2}
\!+\!|S^{(i+1)}(r\!-\!\tau)\!-\!S^{(i)}(r\!-\!\tau)|^{2}
\!+\!\mathbb{W}^{2}_{2}(\mathcal{L}_{r}^{S^{(i)}},\!\mathcal{L}_{r}^{S^{(i-1)}})
\Big)^{\frac{p}{2}}\!\bigg\}\mathrm{d}r\!\bigg]\nn\\
&\leq C\int_{0}^{t}\mathbb{E}\Big[\sup_{0\leq v\leq r}\big|S^{(i+1)}(v)-S^{(i)}(v)\big|^{p}\Big]\mathrm{d}r
+Ct\mathbb{E}\Big[\sup_{0\leq r\leq t}\big|S^{(i)}(r)-S^{(i-1)}(r)\big|^{p}\Big] .
\end{align}
Furthermore, using the BDG inequality one can see that
\begin{align}\label{eq3.22}
H_{2} \leq& C\mathbb{E}\Big[\Big(\int_{0}^{t}|\Gamma^{(i)}(r)|^{2p-2}
\big|\sigma(S^{(i+1)}(r),S^{(i+1)}(r-\tau),\mathcal{L}_{r}^{S^{(i)}})\nn\\
& ~~~~~~~~~~~~~~~~~~~~~~~~~~~ -\sigma(S^{(i)}(r),S^{(i)}(r-\tau),\mathcal{L}_{r}^{S^{(i-1)}})\big|^{2}\mathrm{d}r\Big)^{\frac{1}{2}}\Big]\nn\\
\leq &C\mathbb{E}\Big[\sup_{0\leq r\leq t}|\Gamma^{(i)}(r)|^{p-1}
   \Big(\int_{0}^{t}
\big|\sigma(S^{(i+1)}(r),S^{(i+1)}(r-\tau),\mathcal{L}_{r}^{S^{(i)}})\nn\\
& ~~~~~~~~~~~~~~~~~~~~~~~~~~~-\sigma(S^{(i)}(r),S^{(i)}(r-\tau),\mathcal{L}_{r}^{S^{(i-1)}})\big|^{2}\mathrm{d}r\Big)^{\frac{1}{2}}\Big]\nn\\
\leq&\frac{1}{2}\mathbb{E}\big[\sup_{0\leq r\leq t}|\Gamma^{(i)}(r)|^{p}\big] +C\mathbb{E}\Big[\int_{0}^{t}\Big(\big|S^{(i+1)}(r)-S^{(i)}(r)\big|^{2}
\nn\\
&~~~~~~~~~~~~~~~~~~~~~~~~~~~~~~~~+\big|S^{(i+1)}(r-\tau)-S^{(i)}(r-\tau)\big|^{2}+\mathbb{W}^{2}_{2}(\mathcal{L}_{r}^{S^{(i)}},\mathcal{L}_{r}^{S^{(i-1)}})\Big)^{\frac{p}{2}}\mathrm{d}r\Big]\nn\\
 \leq&\frac{1}{2}\mathbb{E}\big[\sup_{0\leq r\leq t}|\Gamma^{(i)}(r)|^{p}\big] +C\int_{0}^{t}\mathbb{E}\Big[\sup_{0\leq v\leq r}\big|S^{(i+1)}(v)-S^{(i)}(v)\big|^{p}\Big]\mathrm{d}r
 \nn\\
&~~~~~~~~~~~~~~~~~~~~~~~~~~~~~~~~+Ct\mathbb{E}\Big[\sup_{0\leq r\leq t}\big|S^{(i)}(r)-S^{(i-1)}(r)\big|^{p}\Big] .
\end{align}
Substituting~$(\ref{eq3.21})$~and~$(\ref{eq3.22})$~into~$(\ref{eq3.20})$~arrives at
\begin{align}\label{eq3.23}
&\mathbb{E}\big[\sup_{0\leq r\leq t}|\Gamma^{(i)}(r)|^{p}\big]\nn\\
&\leq C\int_{0}^{t}\mathbb{E}\Big[\sup_{0\leq v\leq r}\big|S^{(i+1)}(v)-S^{(i)}(v)\big|^{p}\Big]\mathrm{d}r+Ct\mathbb{E}\Big[\sup_{0\leq r\leq t}\big|S^{(i)}(r)-S^{(i-1)}(r)\big|^{p}\Big].
\end{align}
By~$(\ref{eq3.19})$, $(\ref{eq3.23})$~and the Gronwall inequality, we obtain that
\begin{align}\label{eq3.24}
\mathbb{E}\Big[\sup_{0\leq r\leq t}\big|S^{(i+1)}(r)-S^{(i)}(r)\big|^{p}\Big]\leq Ct e^{Ct}\mathbb{E}\Big[\sup_{0\leq r\leq t}\big|S^{(i)}(r)-S^{(i-1)}(r)\big|^{p}\Big].
\end{align}
Taking~$T_{0}\in (0, T]$~such that~$C_2 T_{0} e^{C_1 T_{0}}\leq e^{-1}$, by~$(\ref{eq3.24})$, we obtain that
\begin{align*}
 \mathbb{E}\Big[\sup_{0\leq r\leq T_{0}}\big|S^{(i+1)}(r)-S^{(i)}(r)\big|^{p}\Big]
&\leq e^{-1}\mathbb{E}\Big[\sup_{0\leq r\leq T_{0}}\big|S^{(i)}(r)-S^{(i-1)}(r)\big|^{p}\Big]\nn\\
&\leq \cdots \nn\\
&\leq e^{-i}\mathbb{E}\Big[\sup_{0\leq r\leq T_{0}}\big|S^{(1)}(r)-S^{(0)}(r)\big|^{p}\Big]\nn\\
&\leq e^{-i}2^{p-1}\Big(\mathbb{E}\big[\sup_{0\leq r\leq T_{0}} |S^{(1)}(r)|^{p}\big]+\mathbb{E}[|\xi(0)|^{p}]\Big)\nn\\
&\leq Ce^{-i}.
\end{align*}
Then, there is an~$\mathcal{F}_{t}$-adapted continuous process~$\{S(t)\}_{t\in[0,T_{0}]}$~such that
\begin{align}\label{eq3.26}
\lim_{i\rightarrow\infty}\sup_{0\leq t\leq T_{0}}\mathbb{W}^{p}_{p}(\mathcal{L}_{t}^{S^{(i)}},\mathcal{L}_{t}^{S})
\leq\lim_{i\rightarrow\infty}\mathbb{E}\Big[\sup_{0\leq t\leq T_{0}}\big|S^{(i)}(t)-S(t)\big|^{p}\Big]=0,
\end{align}
where~$\mathcal{L}_{t}^{S}$~is the law of~$S(t)$.
And~$(\ref{eq3.26})$~implies~$\mathbb{E}[\sup\limits_{0\leq t\leq T_{0}}|S(t)|^{p}]<\infty$. Taking limits on both sides of~$(\ref{eq3.9})$~we derive that~$\mathbb{P}$-$a.s.$
\begin{equation*}
\mathrm{d}\Big(S(t)-D(S(t-\tau))\Big)=
b(S(t),S(t-\tau),\mathcal{L}_{t}^{S})\mathrm{d}t+\sigma(S(t),S(t-\tau),\mathcal{L}_{t}^{S})\mathrm{d}B(t),~~~t\in[0,T_{0}].
\end{equation*}
Taking~$T_{0}$~as the initial time and repeating the  previous procedure, one concludes that~$(\ref{eq3.1})$~has a solution~$\{S(t)\}_{t\in[0,T]}$~with the property
$$
\mathbb{E}[\sup\limits_{0\leq t\leq T}|S(t)|^{p}]<\infty.
$$
It still needs to prove the uniqueness. In fact, assume that~$\{S(t)\}_{t\in[0,T]}$~and~$\{\bar{S}(t)\}_{t\in[0,T]}$~are two solutions to~$(\ref{eq3.1})$ with the same initial data, using the similar way as~$(\ref{eq3.24})$, it is obvious to see
$$
\mathbb{E}\Big[\sup_{0\leq t\leq T}\big|S(t)-\bar{S}(t)\big|^{p}\Big]=0.
$$
\underline{Step 3.} We prove that \eqref{eq3.8} holds. Similar to ~$(\ref{eq3.12})$, we have
\begin{align} \label{eq3.28}
\sup_{0\leq r\leq t}|S(r)|^{p}\leq\frac{\lambda}{1-\lambda}\|\xi\|^{p}+\frac{1}{(1-\lambda)^{p}}\sup_{0\leq r\leq t}|S(r)-D(S(r-\tau))|^{p}.
\end{align}
The~It$\hat{\mathrm{o}}$~formula leads to
\begin{align} \label{eq3.29}
&E\Big[\sup_{0\leq r\leq t}\big|S(r)-D(S(r-\tau))\big|^{p}\Big] \leq E\Big[\big|\xi(0)-D(\xi(-\tau))\big|^{p}\Big]+ \bar{J}_{1}+\bar{J}_{2},
\end{align}
where
\begin{align*}
 \bar{J}_{1}=&pE\Big[\sup_{0\leq r\leq t}\int_{0}^{r}\big|S(v)-D(S(v-\tau))\big|^{p-2}
 \times\Big(\big(S(v)-D(S(v-\tau))\big)^{T}\\
&~~~~~~~~~~~~~~~~\times  b(S(v),S(v-\tau),\mathcal{L}_{v}^{S})+\frac{ p-1 }{2} \big|\sigma(S(v),S(v-\tau),\mathcal{L}_{v}^{S})\big|^{2}\Big)\mathrm{d}v\Big]\\
\end{align*}
and
\begin{align*}
 \bar{J}_{2}=&pE\Big[\sup_{0\leq r\leq t}\int_{0}^{r}\big|S(v)-D(S(v-\tau))\big|^{p-2}\\
 &~~~~~~~~~~~~~~~~~~~~~~~~~\times\big(S(v)-D(S(v-\tau))\big)^{T} \sigma(S(v),S(v-\tau),\mathcal{L}_{v}^{S})\mathrm{d}B(v)\Big].
\end{align*}
By~$(\ref{eq3.3})$~and~$(\ref{eq3.5})$, we obtain
\begin{align}\label{eq3.30}
\bar{J}_{1}&\leq C\mathbb{E}\Big[ \int_{0}^{t}\big|S(r)-D(S(r-\tau))\big|^{p-2}
\Big(1+|S(r)|^{2}+|S(r-\tau)|^{2}+W_{2}^{2}(\mathcal{L}_{r}^{S})\Big)\mathrm{d}r\Big] \nn\\
&\leq C\mathbb{E}\Big[\int_{0}^{t}\Big(1+|S(r)|^{p}+|S(r-\tau)|^{p}+W_{p}^{p}(\mathcal{L}_{r}^{S})\Big)\mathrm{d}r\Big]\nn\\
&\leq C\int_{0}^{t}\mathbb{E}\Big[\big(1+|S(r)|^{p}+|S(r-\tau)|^{p}\big)\Big]\mathrm{d}r\nn\\
&\leq C T+C\int_{0}^{t}\mathbb{E}\big[\sup_{0\leq v\leq r}|S(v)|^{p}\big]\mathrm{d}r+C\int_{0}^{t}\mathbb{E}\big[\sup_{0\leq v\leq r}|S(v-\tau)|^{p}\big]\mathrm{d}r \nn\\
&\leq C T\big(1+\mathbb{E}[\|\xi\|^{p}]\big)+C\int_{0}^{t}\mathbb{E}\big[\sup_{0\leq v\leq r}|S(v)|^{p}\big]\mathrm{d}r .
\end{align}
Using the BDG inequality, Young's inequality and~$\mathrm{H\ddot{o}lder}$'s inequality, we have
\begin{align}\label{eq3.31}
\bar{J}_{2}&\leq C\mathbb{E}\Big[\Big(\int_{0}^{t}\big|S(r)-D(S(r-\tau))\big|^{2p-2}
|\sigma(S(r),S(r-\tau),\mathcal{L}_{r}^{S})|^{2}\mathrm{d}r\Big)^{\frac{1}{2}}\Big]\nn\\
&\leq C\mathbb{E}\Big[\sup_{0\leq r\leq t}\big|S(r)-D(S(r-\tau))\big|^{p-1}
\big(\int_{0}^{t}|\sigma(S(r),S(r-\tau),\mathcal{L}_{r}^{S})|^{2}\mathrm{d}r\big)^{\frac{1}{2}}\Big]\nn\\
&\leq\frac{1}{2}\mathbb{E}\Big[\sup_{0\leq r\leq t}\big|S(r)-D(S(r-\tau))\big|^{p}\Big]\nn\\
&~~~~+C\mathbb{E}\bigg[\bigg(\int_{0}^{t}\Big(1+|S(r)|^{2}+|S(r-\tau)|^{2}+W^{2}_{2}(\mathcal{L}_{r}^{S})\Big)\mathrm{d}r\bigg)^{\frac{p}{2}}\bigg]\nn\\
&\leq\frac{1}{2}\!\mathbb{E}\Big[\!\sup_{0\leq r\leq t}\big|S(r)-\!D(S(r\!-\!\tau))\big|^{p}\Big]
\!+\!C\mathbb{E}\Big[\int_{0}^{t}\Big(1\!+|S(r)|^{p}\!+|S(r-\tau)|^{p}\!+W_{p}^{p}(\mathcal{L}_{r}^{S})\Big)\mathrm{d}r\Big]\nn\\
&\leq\frac{1}{2}\mathbb{E}\Big[\!\sup_{0\leq r\leq t}\big|S(r)\!-\!D(S(r\!-\!\tau))\big|^{p}\!\Big]\!+\!CT\big(1\!+\!\mathbb{E}[\|\xi\|^{p}]\big)\!+\!C\int_{0}^{t}\mathbb{E}\big[\sup_{0\leq v\leq r}|S(v)|^{p}\big]\mathrm{d}r.
\end{align}
Inserting~$(\ref{eq3.30})$~and~$(\ref{eq3.31})$~into~$(\ref{eq3.29})$~yields
\begin{align*}
\mathbb{E}\Big[\sup_{0\leq r\leq t}\big|S(r)-D(S(r-\tau))\big|^{p}\Big]\leq C\int_{0}^{t}\mathbb{E}\big[\sup_{0\leq v\leq r}|S(v)|^{p}\big]\mathrm{d}r+CT\big(1+\mathbb{E}[\|\xi\|^{p}]\big),
\end{align*}
which together with~$(\ref{eq3.28})$~and the Gronwall inequality means that
\begin{align*}
\mathbb{E}\big[\sup_{0\leq r\leq t}|S(r)|^{p}\big]\leq V(1+T)e^{V(1+T)T},
\end{align*}
where~$V$ denotes a constant which depends on $p,K_{1},K_{2},K_{3}$ and $\mathbb{E}[\|\xi\|^{p}]$.\end{proof}

\subsection{Stochastic Particle Method}\label{s3.2}

In this subsection, we take use of stochastic particle method \cite{TD1996, BM1997} to approximate MV-NSDDE~$(\ref{eq3.1})$-$(\ref{eq3.2})$. For any~$ a\in\mathbb{N}$,
$\{B^{a}(t)\}_{t\in[0,T]}$~is~an $m$-dimension Brownian motion and $\{B^{1}(t)\}, \{B^{2}(t)\}, \cdots$~are  independent.
$\xi^{a}\in L^{p}_{\mathcal{F}_{0}}([-\tau,0];\mathbb{R}^{d})$,   $\xi^{1}, \xi^{2}, \cdots$ are independent with the identity distribution ($i.i.d.$) and  $\mathbb{E}[\sup_{_{-\tau\leq r,t\leq0}}|\xi^a(t)-\xi^a(r)|^{p}]\leq K_0|t-r|^{\frac{p}{2}}$.

 Let~$\{S^a(t)\}_{t\in[0,T]}$~denote the unique solution to MV-NSDDE
\begin{equation} \label{eq3.32}
\mathrm{d}\Big(S^a(t)-D(Y^{a}(t-\tau))\Big)=b(S^a(t),S^{a}(t-\tau),\mathcal{L}_{t}^{S^{a}})
\mathrm{d}t+\sigma(S^a(t),S^{a}(t-\tau),\mathcal{L}_{t}^{S^{a}})B^{a}(t),
\end{equation}
with the initial condition~$S^{a}_0 =\xi^{a} $, where~$\mathcal{L}_{t}^{S^{a}}$~denotes the law of~$ S^a(t)$.  We can see  that $ S^{1}(t), S^{2}(t), \cdots$ are $i.i.d.$ for $t\geq 0$.

For each  $\Xi\in\mathbb{N},~1\leq a\leq \Xi$, let~$ S^{a,M}(t) $~be the solution of NSDDE
\begin{align} \label{eq3.33}
\mathrm{d}\Big(S^{a,\Xi}(t)-D(S^{a,\Xi}(t-\tau))\Big)
=&b(S^{a,\Xi}(t),S^{a,\Xi}(t-\tau),\mathcal{L}_{t}^{S,\Xi})\mathrm{d}t\nn\\
&\!+\!\sigma(S^{a,\Xi}(t),S^{a,\Xi}(t\!-\!\tau),\mathcal{L}_{t}^{S,\Xi})\mathrm{d}B^{a}(t),~~~t\in[0,T],
\end{align}
with the initial condition~$ S^{a,\Xi}_0 =\xi^{a} $, where
$\mathcal{L}_{t}^{S,\Xi}(\cdot):=\frac{1}{\Xi}\sum\limits_{j=1}^{\Xi}\delta_{S^{j,\Xi}(t)}(\cdot)$.
We prepare a path-wise propagation of chaos result on NSDDEs~$(\ref{eq3.33})$.

\begin{lemma}\label{le3.6}
If Assumption~$\ref{a1}$~holds and  $p>4$ in \eqref{eq3.2}, then
\begin{align*}
\displaystyle\sup_{1\leq a\leq \Xi}\mathbb{E}\big[\sup_{0\leq t\leq T}|S^a(t)-S^{a,\Xi}(t)|^{2}\big]\leq C\left\{
\begin{array}{lll}
\Xi^{-1/2},~~~&1\leq d<4,\\
\Xi^{-1/2}\log(\Xi),~~~&d=4,\\
\Xi^{-d/2},~~~&4<d,
\end{array}
\right.
\end{align*}
where~$C$~depends on the constant on the right side of \eqref{eq3.8} but is independent of $M$.
\end{lemma}
\begin{proof}
For any~$1\leq a\leq \Xi$~and~$t\in[0,T]$, using Lemma~$\ref{l3.4}$, it easy to derive that
\begin{align} \label{eq3.34}
\sup_{0\leq r\leq t}|S^{a}(r)-S^{a,\Xi}(r)|^{2}
\leq\frac{1}{(1-\lambda)^{2}}\sup_{0\leq r\leq t}|\Gamma^{a}(r) |^{2},
\end{align}
where
$
\Gamma^{a}(t)=S^a(t)-D(S^{a}(t-\tau))-S^{a,\Xi}(t)+D(S^{a,\Xi}(t-\tau)).
$
It follows from It$\hat{\mathrm{o}}$~formula that
\begin{align} \label{eq3.35}
\mathbb{E}\big[\sup_{0\leq r\leq t}|\Gamma^{a}(r)|^{2}\big]\leq  Q_{1}+Q_{2},
\end{align}
where
\begin{align*}
 Q_{1}=&\mathbb{E}\bigg[\sup_{0\leq r\leq t}  \int_{0}^{r}\Big\{2 \big(\Gamma^{a}(v)\big)^{T} \Big(b(S^{a}(v),S^{a}(v-\tau),\mathcal{L}_{v}^{S^{a}})-
 b(S^{a,\Xi}(v),S^{a,\Xi}(v-\tau),\mathcal{L}_{v}^{S,\Xi}) \Big)  \\
 &~~~~~~~~~~~~~~~~~~ +\big|\sigma(S^{a}(v),S^{a}(v-\tau),\mathcal{L}_{v}^{S^{a}})-\sigma(S^{a,\Xi}(v),S^{a,\Xi}(v-\tau),\mathcal{L}_{v}^{S,\Xi})\big|^{2}\Big\}\mathrm{d}v\bigg]\\
 \end{align*}
 and
\begin{align*}
Q_{2}=&\mathbb{E}\bigg[\sup_{0\leq r\leq t}\int_{0}^{r}\!2\big(\Gamma^{a}(v)\big)^{T}\!\Big(\sigma(S^{a}(v),S^{a}(v\!-\!\tau),\mathcal{L}_{v}^{S^{a}})\!\\
 &~~~~~~~~~~~~~~~~~~~~~~~~~~~~~~~~-\sigma(S^{a,\Xi}(v),S^{a,\Xi}(v\!-\!\tau),\mathcal{L}_{v}^{S,\Xi})\!\Big)\mathrm{d}B^a(v)\bigg].
\end{align*}
By Assumption~$\ref{a1}$~and Young's inequality, we have
\begin{align}\label{eq3.36}
Q_{1}&\leq C\mathbb{E}\bigg[\int_{0}^{t}
\Big(\big|S^{a}(r)-S^{a,\Xi}(r)\big|^{2}+\big|S^{a}(r-\tau)-S^{a,\Xi}(r-\tau)\big|^{2}
+\mathbb{W}^{2}_{2}(\mathcal{L}_{r}^{S^{a}},\mathcal{L}_{r}^{S,\Xi})\Big)\mathrm{d}r\bigg]\nn\\
&\leq C \int_{0}^{t}\mathbb{E}\Big[\sup_{0\leq v\leq r}\big|S^{a}(v)-S^{a,\Xi}(v)\big|^{2}\Big]\mathrm{d}r
+C \int_{0}^{t}\mathbb{E}\big[\mathbb{W}^{2}_{2}(\mathcal{L}_{r}^{S^{a}},\mathcal{L}_{r}^{S,\Xi})\big]\mathrm{d}r.
\end{align}
Using the~BDG~inequality and~$\mathrm{H\ddot{o}lder}$'s~inequality, we then have
\begin{align}\label{eq3.37}
Q_{2}&\leq C\mathbb{E}\bigg[\Big(\int_{0}^{t}|\Gamma^{a}(r)|^{2 }
\big|\sigma(S^{a}(r),S^{a}(r-\tau),\mathcal{L}_{r}^{S^{a}})-\sigma(S^{a,\Xi}(r),S^{a,\Xi}(r-\tau),\mathcal{L}_{r}^{S,\Xi})\big|^{2}\mathrm{d}r\Big)^{\frac{1}{2}}\bigg]\nn\\
&\leq C\mathbb{E}\bigg[\!\sup_{0\leq r\leq t}|\Gamma^{a}(r)| \Big(\int_{0}^{t}
\big|\sigma(S^{a}(r),S^{a}(r\!-\!\tau),\!\mathcal{L}_{r}^{S^{a}})\!-\!\sigma(S^{a,\Xi}(r),S^{a,\Xi}(r\!-\!\tau),\mathcal{L}_{r}^{S,\Xi})\big|^{2}\mathrm{d}r\Big)^{\frac{1}{2}}\!\bigg]\nn\\
&\leq\frac{1}{2}\mathbb{E}\big[\sup_{0\leq r\leq t}|\Gamma^{a}(r)|^{2}\big]\nn\\
&~~~~+C\mathbb{E}\Big[\int_{0}^{t}\Big(\big|S^{a}(r)\!-\!S^{a,\Xi}(r)\big|^{2}
+\big|S^{a}(r-\tau)-S^{a,\Xi}(r-\tau)\big|^{2}+\mathbb{W}^{2}_{2}(\mathcal{L}_{r}^{S^{a}},\mathcal{L}_{r}^{S,\Xi})\Big)\mathrm{d}r\Big]\nn\\
&\leq\frac{1}{2}\mathbb{E}\big[\sup_{0\leq r\leq t}|\Gamma^{a}(r)|^{2}\big]\nn\\
&~~~~+C\int_{0}^{t}\mathbb{E}\big[\sup_{0\leq v\leq r}\big|S^{a}(v)-S^{a,\Xi}(v)\big|^{2}\big]\mathrm{d}r+C\int_{0}^{t}\mathbb{E}\Big[\mathbb{W}^{2}_{2}(\mathcal{L}_{r}^{S^{a}},\mathcal{L}_{r}^{S,\Xi})\Big]\mathrm{d}r.
\end{align}
Now we introduce another empirical measure constructed from the exact solution to~$(\ref{eq3.32})$~by
$$
\mathcal{L}_{r}^{\Xi}( dx)=\frac{1}{\Xi}\sum_{j=1}^{\Xi}\delta_{S^{j}(r)}(dx).
$$
One notices
\begin{align*}
\mathbb{W}^{2}_{2}(\mathcal{L}_{r}^{S^{a}},\mathcal{L}_{r}^{S,\Xi})&\leq 2 \big(\mathbb{W}^{2}_{2}(\mathcal{L}_{r}^{S^{a}},\mathcal{L}_{r}^{\Xi})
+\mathbb{W}^{2}_{2}(\mathcal{L}_{r}^{\Xi},\mathcal{L}_{r}^{S,\Xi})\big)\nn\\
&= 2 \mathbb{W}^{2}_{2}(\mathcal{L}_{r}^{S^{a}},\mathcal{L}_{r}^{\Xi})+
\frac{2}{\Xi}\sum_{j=1}^{\Xi}\big|S^{j}(r)-S^{j,\Xi}(r)\big|^{2} ,
\end{align*}
and
\begin{align*}
\mathbb{E}\Big[\frac{1}{\Xi}\sum_{j=1}^{\Xi}\big|S^{j}(r)-S^{j,M}(r)\big|^{2}\Big]
=\mathbb{E}\Big[\big|S^{a}(r)-S^{a,\Xi}(r)\big|^{2}\Big].
\end{align*}
Combining the above inequalities arrives at
\begin{align}\label{eq3.38}
\int_{0}^{t}\!\mathbb{E}\big[\mathbb{W}^{2}_{2}(\mathcal{L}_{r}^{S^{a}},\!\mathcal{L}_{r}^{S,\Xi})\big]\mathrm{d}r
\leq 2\!\int_{0}^{t}\!\mathbb{E}\Big[\!\sup_{0\leq v\leq r}\big|S^{a}(v)\!-\!S^{a,\Xi}(v)\big|^{2}\Big]\mathrm{d}r
\!+\!2\!\int_{0}^{t}\!\mathbb{E}\big[\!\mathbb{W}^{2}_{2}(\mathcal{L}_{r}^{S^{a}},\!\mathcal{L}_{r}^{\Xi})\big]\mathrm{d}r.
\end{align}
This together with~$(\ref{eq3.34})$-$(\ref{eq3.37})$~yields
\begin{align*}
\mathbb{E}\big[\sup_{0\leq r\leq t}|S^{a}(r)-S^{a,\Xi}(r)|^{2}\big]&\leq \frac{1}{(1-\lambda)^{2}}\mathbb{E}\big[\sup_{0\leq r\leq t}|\Gamma^{a}(r)|^{2}\big]\nn\\
&\leq C\!\int_{0}^{t}\mathbb{E}\Big[\sup_{0\leq v\leq r}\!\big|S^{a}(v)\!-\!S^{a,\Xi}(v)\big|^{2}\Big]\mathrm{d}r
\!+\!C\int_{0}^{t}\!\mathbb{E}\big[\mathbb{W}^{2}_{2}(\mathcal{L}_{r}^{S^{a}},\!\mathcal{L}_{r}^{\Xi})\big]\mathrm{d}r.
\end{align*}
Therefore, thanks to~$(\ref{eq3.8})$~and Lemma~$\ref{l3.2}$, it follows from the~Gronwall~inequality that
\begin{align*}
\mathbb{E}\big[\sup_{0\leq t\leq T}|S^a(t)-S^{a,\Xi}(t)|^{2}\big]\leq C\left\{
\begin{array}{lll}
\Xi^{-1/2},~~~&1\leq d<4,\\
\Xi^{-1/2}\log(\Xi),~~~&d=4,\\
\Xi^{-d/2},~~~&4<d.
\end{array}
\right.
\end{align*}
\end{proof}

\section{ Tamed EM Method of MV-NSDDEs}\la{cov-rates}
We propose in this section an appropriate explicit method for approximating the solution of~$(\ref{eq3.33})$, and go a further step to give the strong convergence between the numerical solutions and the exact solution of~$(\ref{eq3.33})$. By virtue of propagation of chaos, we establish the result of convergence of the numerical approximation for the original MV-NSDDE~$(\ref{eq3.32})$.

Since the drift coefficients satisfying Assumption $1$ and  Assumption $2$ might be nonlinear, for any $\triangle\in (0,1\wedge \tau)$, we define an auxiliary function
\begin{align}\label{eq4.1}
b_{\triangle}(x,y,\mu)
=\frac{b(x,y,\mu)}{1+\triangle^{\alpha}|b(x,y,\mu)|},
\end{align}
for~$x,y\in\RR^{d}$,~$\mu\in\mathcal{P}_{2}(\RR^{d})$~and~$\alpha\in(0,\frac{1}{2}]$.
To approximate $(\ref{eq3.33})$
we propose the tamed EM scheme. In addition, we let $\frac{\tau}{\triangle}=: n_0$ be an integer (for example, when $\tau$ and $\triangle$ are rational numbers). For the fixed step size~$\triangle$, define
\be\label{eq4.2}\left\{\begin{array}{lcl}
U^{a,\Xi}_{n}=\xi^a(t_{n}),~~~~~~~~~~~a=1,\cdots,\Xi,~~n=-n_0,\cdots,0,\\
U^{a,\Xi}_{n+1}-D(U^{a,\Xi}_{n+1-n_0})= U^{a,\Xi}_{n}-D(U^{a,\Xi}_{n-n_0})
+b_{\triangle}(U^{a,\Xi}_{n},U^{a,\Xi}_{n-n_0},\mathcal{L}^{U_n,\Xi} )\triangle\\
~~~~~~~~~~~~~~~~~~~~~~~~~~~~+\sigma(U^{a,\Xi}_{n},U^{a,\Xi}_{n-n_0},\mathcal{L}^{U_n,\Xi} )\triangle B^{a}_{n},~~~~~n\geq 0,
\end{array}\right.\ee
where  $t_{-n_0}=-\tau$, $t_{n}=n\triangle$ for $n>-n_0$,  $\mathcal{L}^{U_n,\Xi} =\frac{1}{\Xi}\sum\limits_{j=1}^{\Xi}\delta_{U^{j,\Xi}_{n}}$~and~$\triangle B^{a}_{n}=B^{a}(t_{n+1})-B^{a}(t_{n})$. To proceed, we give the  numerical solution of this scheme
\begin{align}\label{eq4.3}
U^{a,\Xi}(t)=U^{a,\Xi}_{n}, ~~~~ t\in[t_n, t_{n+1}).
\end{align}
For convenience, we define $\mathcal{L}^{U,\Xi}_{t}=\frac{1}{\Xi}\sum\limits_{j=1}^{\Xi}\delta_{U^{j,\Xi}(t)}$ and $\rho(t)=\lfloor t/\triangle\rfloor\triangle$~for~$t\geq -\tau$. Then one observes~$\mathcal{L}^{U,\Xi}_{t}=\mathcal{L}^{U,\Xi}_{\rho(t)}=\mathcal{L}^{U_n,\Xi}$, for $t\in [t_n, t_{n+1})$.
Moreover, we give the auxiliary continuous numerical solution
\be \label{eq4.4}\left\{\begin{array}{lll}
 \bar{U}^{a,\Xi}(t) = \xi^{a}(t),&~~ -\tau\leq t\leq 0,\\
\bar{U}^{a,\Xi}(t)\!-\!D\big(\!\bar{U}^{a,\Xi}( t\!-\!\tau )\big)\!=&\!\xi^{a}(0)\!-\!D(\xi^{a}(\!-\tau))\!+\!\int^{t}_{0}\!b_{\triangle}(U^{a,\Xi}(r),\!U^{a,\Xi}(r\!-\!\tau ),\!\mathcal{L}_{r}^{U,\Xi})\!\mathrm{d}r\\
&+\!\int^{t}_{0}\!\sigma (U^{a,\Xi}(r),\!U^{a,\Xi}(r\!-\!\tau ),\!\mathcal{L}_{r}^{U,\Xi})\mathrm{d}B^{a}(r),~~~~ t>0.
\end{array}\right.\ee
It follows from~$(\ref{eq4.4})$~directly that for $t>0$,
\begin{align*}
&\bar{U}^{a,\Xi}(t)-D\big(\bar{U}^{a,\Xi}( t-\tau )\big)\nn\\
&=U^{a,\Xi}_0-D\big(U^{a,\Xi}_{ -n_0}\big)+\sum_{n=0}^{\lfloor t/\triangle\rfloor-1}\int^{t_{n+1}}_{t_n}b_{\triangle}(U^{a,\Xi}_n,U^{a,\Xi}_{n-n_0},\mathcal{L}^{U_n,\Xi})\mathrm{d}r\nn\\
&~~~~+\sum_{n=0}^{\lfloor t/\triangle\rfloor-1}\int^{t_{n+1}}_{t_n}\sigma (U^{a,\Xi}_n,U^{a,\Xi}_{n-n_0},\mathcal{L}^{U_n,\Xi})\mathrm{d} B^{a}(r)\nn\\
&~~~~+\int^{t}_{\rho(t)}b_{\triangle}(U^{a,\Xi}(r),U^{a,\Xi}(r-\tau),\mathcal{L}_{r}^{U,\Xi})\mathrm{d}r+\int^{t}_{\rho(t)}\sigma(U^{a,\Xi}(r),U^{a,\Xi}(r-\tau),\mathcal{L}_{r}^{U,\Xi})\mathrm{d}B^{a}(r)\nn\\
&=U^{a,\Xi}(t)-D\big( {U}^{a,\Xi}( t-\tau )\big)+\int^{t}_{\rho(t)}b_{\triangle}(U^{a,\Xi}(r),U^{a,\Xi}(r-\tau),\mathcal{L}_{r}^{U,\Xi})\mathrm{d}r\nn\\
&~~~~+\int^{t}_{\rho(t)}\sigma(U^{a,\Xi}(r),U^{a,\Xi}(r-\tau),\mathcal{L}_{r}^{U,\Xi})\mathrm{d}B^{a}(r).
\end{align*}
One observes from the above equation that
$$\bar{U}^{a,\Xi}(\rho(t))-D\big(\bar{U}^{a,\Xi}(  \rho(t)-\tau )\big) =U^{a,\Xi}(\rho(t))-D\big( {U}^{a,\Xi}( \varrho(t)-\tau  )\big)=U^{a,\Xi}(t)-D\big( {U}^{a,\Xi}( t-\tau )\big).$$
Due to solving the difference equations we arrive at
\begin{align}\label{eq4.5}
 \bar{U}^{a,\Xi}(\rho(t))  =U^{a,\Xi}(\rho(t)) =U^{a,\Xi}(t)=U^{a,\Xi}_n,~~~~t\in[t_n, t_{n+1}).
\end{align}
Therefore, it is clearly that~$\bar{U}^{a,\Xi}(t)$~and~$U^{a,\Xi}(t)$~take  the same values at the grid points.
\begin{remark}\label{r4.1}
From~$(\ref{eq4.1})$ and \eqref{eq3.5}, one observes
\begin{align}\label{eq4.6}
|b_{\triangle}(x,y,\mu)|\leq \triangle^{-\alpha} \wedge |b(x,y,\mu)|,
\end{align}and
\begin{align}\label{eq4.7}
 \big(x-D(y)\big)^{T}b_{\triangle}(x,y,\mu) \leq(K_1+  K_2) \big( 1+|x|^{2}+|y|^{2}+W^{2}_{2}(\mu) \big).
\end{align}
\end{remark}

\subsection{ Moment estimate}\label{bound}
For~$1\leq a\leq \Xi$, in order to show that the numerical solution produced by~$(\ref{eq4.2})$ converge to the solution of $(\ref{eq3.33})$, in this subsection we establish boundedness of $p$th moment of the numerical solution~$\bar{U}^{a,\Xi}(t)$ of~$(\ref{eq4.4})$.
\begin{lemma}\label{le4.2}
Under Assumption~$\ref{a1}$,
$$
\sup_{1\leq a\leq \Xi}\mathbb{E}\big[\sup_{0\leq t\leq T}|\bar{U}^{a,\Xi}(t)|^{p}\big]\leq C,~~~T\geq0,
$$
where ~$C$~is the constant not only its value may vary with different lines but also is independent of~$\Xi$~and~$\triangle$ from now on.
\end{lemma}
\begin{proof}
Let~$ T\geq0$, for each~$a=1,\cdot\cdot\cdot,\Xi$~and positive integer~$N$, let us define the~$\mathcal{F}_{t}$-stopping time by
$$
\eta^{a,\Xi}_{N}=T\wedge\inf\{t\in[0,T]:~|\bar{U}^{a,\Xi}(t)|\geq N\},~~~~\eta^{\Xi}_{N}=\min_{1\leq a\leq \Xi}\eta^{a,\Xi}_{N},
$$
and~$\eta^{\Xi}_{N}\uparrow T$~as~$N\rightarrow\infty$. Due to~$(\ref{eq3.7})$ and Assumption \ref{a1}, we know that
\begin{align*}
 |\bar{U}^{a,\Xi}(t\wedge\eta^{\Xi}_{N})|^{p} \leq\lambda\big|\bar{U}^{a,\Xi}(t\wedge\eta^{\Xi}_{N}-\tau)\big|^{p}+\frac{1}{(1-\lambda)^{p-1}} \big|\bar{U}^{a,\Xi}(t\wedge\eta^{\Xi}_{N})-D\big(\bar{U}^{a,\Xi}(t\wedge\eta^{\Xi}_{N}-\tau)\big)\big|^{p},
\end{align*}
which implies
\begin{align}\label{eq4.8}
 \sup_{0\leq r\leq t}|\bar{U}^{a,\Xi}(r\wedge\eta^{\Xi}_{N})|^{p}
\leq&\lambda\sup_{-\tau\leq r\leq 0}|\bar{U}^{a,\Xi}(r)|^{p}+\lambda\sup_{0\leq r\leq t }|\bar{U}^{a,\Xi}(r\wedge\eta^{\Xi}_{N})|^{p}\nn\\
&+\frac{1}{(1-\lambda)^{p-1}}\sup_{0\leq r\leq t}\big|\bar{U}^{a,\Xi}(r\wedge\eta^{\Xi}_{N})-D\big(\bar{U}^{a,\Xi}(r\wedge\eta^{\Xi}_{N}-\tau)\big)\big|^{p}.
\end{align}
Taking expectation on both sides derives that
\begin{align}\label{eq4.9}
&\mathbb{E}\big[\sup_{0\leq r\leq t}|\bar{U}^{a,\Xi}(r\wedge\eta^{\Xi}_{N})|^{p}\big]\nn\\
&\leq\frac{\lambda}{1-\lambda}\mathbb{E}\big[\|\xi^{a}\|^{p}\big]+
\frac{1}{(1-\lambda)^{p}}\mathbb{E}\Big[\sup_{0\leq r\leq t}\big|\bar{U}^{a,\Xi}(r\wedge\eta^{\Xi}_{N})-D\big(\bar{U}^{a,\Xi}(r\wedge\eta^{\Xi}_{N}-\tau)\big)\big|^{p}\Big].
\end{align}
By applying the~It$\hat{\mathrm{o}}$~formula we have
\begin{align}\label{eq4.11}
\mathbb{E}\Big[\sup_{0\leq r\leq t}\big|\bar{U}^{a,\Xi}(r\wedge\eta^{\Xi}_{N})\!-\!D\big(\bar{U}^{a,\Xi}(r\wedge\eta^{\Xi}_{N}\!-\!\tau)\big)\big|^{p}\Big]
\leq\mathbb{E}\big[|\xi^{a}(0)\!-\!D(\xi^{a}(-\tau))|^{p}\big]\!+\!\sum_{l=1}^{3}\tilde{J}^{a}_{l},
\end{align}
where
\begin{align*}
\tilde{J}^{a}_{1} &=p\mathbb{E}\Big[\sup_{0\leq r\leq t} \int_{0}^{r\wedge\eta^{\Xi}_{N}}\big|\bar{U}^{a,\Xi}(v)-D\big(\bar{U}^{a,\Xi}(v-\tau)\big)\big|^{p-2}\nn\\
&~~~~~~~~~~\times\Big(\bar{U}^{a,\Xi}(v)-D\big(\bar{U}^{a,\Xi}(v-\tau)\big)-U^{a,\Xi}(v)+D\big(U^{a,\Xi} (v-\tau)\big)\Big)^{T}\\
&~~~~~~~~~~~~~~~~~~\times b_{\triangle}(U^{a,\Xi}(v),U^{a,\Xi}(v-\tau),\mathcal{L}_{v}^{U,\Xi})\mathrm{d}v\Big],\\
 \tilde{J}^{a}_{2} &=p\mathbb{E}\bigg[\sup_{0\leq r\leq t}\bigg\{\int_{0}^{r\wedge\eta^{\Xi}_{N}}
 \big|\bar{U}^{a,\Xi}(v)-D\big(\bar{U}^{a,\Xi}(v-\tau)\big)\big|^{p-2}\Big(U^{a,\Xi}(v)-D\big(U^{a,\Xi}(v-\tau)\big)\Big)^{T}\\
&~~~~~~~\times b_{\triangle}(U^{a,\Xi}(v),U^{a,\Xi}(v-\tau),\mathcal{L}_{v}^{U,\Xi})\mathrm{d}v \\
&~~~~~~~~~+\!\frac{ p-1 }{2}\!\int_{0}^{r\wedge\eta^{\Xi}_{N}}
\big|\bar{U}^{a,\Xi}(v)\!-\!D\big(\bar{U}^{a,\Xi}(v\!-\!\tau)\big)\!\big|^{p-2} \big|\sigma(U^{a,\Xi}(v),\!U^{a,\Xi}(v\!-\!\tau),\!\mathcal{L}_{v}^{U,\Xi})\big|^{2}\mathrm{d}v\bigg\}\!\bigg]\\
\end{align*}
and
\begin{align*}
\tilde{J}^{a}_{3} &
=p\mathbb{E}\bigg[\sup_{0\leq r\leq t}\int_{0}^{r\wedge\eta^{\Xi}_{N}}\big|\bar{U}^{a,\Xi}(v)-D\big(\bar{U}^{a,\Xi}(v-\tau)\big)\big|^{p-2}
\Big(\bar{U}^{a,\Xi}(v)-D\big(\bar{U}^{a,\Xi}(v-\tau)\big)\Big)^{T}\\
&~~~~~~~~~~\times \sigma(U^{a,\Xi}(v),U^{a,\Xi}(v-\tau),\mathcal{L}_{v}^{U,\Xi})\mathrm{d}B^{a}(v)\bigg].
\end{align*}
We begin to compute~$ \tilde{J}^{a}_{1} $. Due to~$(\ref{eq3.3})$~and~$(\ref{eq4.6})$, using H$\ddot{\hbox{o}}$lder's inequality and Young's inequality, we have
\begin{align}\label{eq4.12}
&\tilde{J}^{a}_{1} \leq p\mathbb{E}\bigg[\int_{0}^{t\wedge\eta^{\Xi}_{N}}\big|\bar{U}^{a,\Xi}(r)-D\big(\bar{U}^{a,\Xi}(r-\tau)\big)\big|^{p-2}\nn\\
&~~~~~~\times\Big|\bar{U}^{a,\Xi}(r)-D\big(\bar{U}^{a,\Xi}(r-\tau)\big)-U^{a,\Xi}(r)+D\big(U^{a,\Xi} (r-\tau)\big)\Big|\nn\\
&~~~~~~~~~~~~~~\times\big|b_{\triangle}(U^{a,\Xi}(r),U^{a,\Xi}(r-\tau),\mathcal{L}_{r}^{U,\Xi})\big|\mathrm{d}r\bigg]\nn\\
&\leq C\mathbb{E}\bigg[\int_{0}^{t\wedge\eta^{\Xi}_{N}}\big|\bar{U}^{a,\Xi}(r)-D\big(\bar{U}^{a,\Xi}(r-\tau)\big)\big|^{p-2}\nn\\
&~~~~\times\Big|\int^{r}_{\rho(r)}\!b_{\triangle}(U^{a,\Xi}(v),\!U^{a,\Xi}(v\!-\!\tau),\mathcal{L}_{v}^{U,\Xi})\mathrm{d}v
\!+\!\int^{r}_{\rho(r)}\sigma(U^{a,\Xi}(v),\!U^{a,\Xi}(v\!-\!\tau),\mathcal{L}_{v}^{U,\Xi})\mathrm{d}B^{a}(v)\Big|\nn\\
&~~~~~~~~~~~~~~\times\big|b_{\triangle}(U^{a,\Xi}(r),U^{a,\Xi}(r-\tau),\mathcal{L}_{r}^{U,\Xi})\big|\mathrm{d}r\bigg]\nn\\
&\leq C\mathbb{E}\bigg[\sup_{0\leq r\leq t}\big|\bar{U}^{a,\Xi}(r\wedge\eta^{\Xi}_{N})-D\big(\bar{U}^{a,\Xi}(r\wedge\eta^{\Xi}_{N}-\tau)\big)\big|^{p-2}\nn\\
&~~~~\times\!\int_{0}^{t\wedge\eta^{\Xi}_{N}}\!\Big|\int^{r}_{\rho(r)}\!b_{\triangle}(U^{a,\Xi}(v),\!U^{a,\Xi}(v\!-\!\tau),\!\mathcal{L}_{v}^{U,\Xi}\!)\mathrm{d}r
\!\!+\!\!\int^{r}_{\rho(r)}\!\sigma(U^{a,\Xi}(v),\!U^{a,\Xi}(v\!-\!\tau),\!\mathcal{L}_{v}^{U,\Xi})\!\mathrm{d}B^{a}(v)\Big|\nn\\
&~~~~~~~~~~~~~\times\big|b_{\triangle}(U^{a,\Xi}(r),U^{a,\Xi}(r-\tau),\mathcal{L}_{r}^{U,\Xi})\big|\mathrm{d}r\bigg]\nn\\
&\leq\frac{1}{4}\mathbb{E}\Big[\sup_{0\leq r\leq t}\big|\bar{U}^{a,\Xi}(r\wedge\eta^{\Xi}_{N})-D\big(\bar{U}^{a,\Xi}(r\wedge\eta^{\Xi}_{N}-\tau)\big)\big|^{p}\Big]\nn\\
&~~~~+C\mathbb{E}\bigg[\int_{0}^{t\wedge\eta^{\Xi}_{N}}\bigg(\big|\int^{r}_{\rho(r)}b_{\triangle}(U^{a,\Xi}(v),U^{a,\Xi}(v-\tau),\mathcal{L}_{v}^{U,\Xi})\mathrm{d}v\big|^{\frac{p}{2}}\nn\\
&~~~~~~+\!\big|\int^{r}_{\rho(r)}\!\sigma(U^{a,\Xi}(v),U^{a,\Xi}(v\!-\!\tau),\mathcal{L}_{v}^{U,\Xi})\mathrm{d}B^{a}(v)\big|^{\frac{p}{2}}\bigg)
\big|b_{\triangle}(U^{a,\Xi}(r),U^{a,\Xi}(r\!-\!\tau),\mathcal{L}_{r}^{U,\Xi})\big|^{\frac{p}{2}}\mathrm{d}r\!\bigg]\nn\\
&\leq\frac{1}{4}\mathbb{E}\Big[\sup_{0\leq r\leq t}\big|\bar{U}^{a,\Xi}(r\wedge\eta^{\Xi}_{N})-D\big(\bar{U}^{a,\Xi}(r\wedge\eta^{\Xi}_{N}-\tau)\big)\big|^{p}\Big]\nn\\
&~~~~+C\bigg(\triangle^{p(\frac{1}{2}-\alpha)}+\triangle^{-\frac{\alpha p}{2}}
\mathbb{E}\Big[\int_{0}^{t}\big|\int^{r\wedge\eta^{\Xi}_{N}}_{\rho(r\wedge\eta^{\Xi}_{N})}\sigma(U^{a,\Xi}(v),U^{a,\Xi}(v-\tau),\mathcal{L}_{v}^{U,\Xi})\mathrm{d}B^{a}(v)\big|^{\frac{p}{2}}
\mathrm{d}r\Big]\bigg)\nn\\
&\leq\frac{1}{4}\mathbb{E}\Big[\sup_{0\leq r\leq t}|\bar{U}^{a,\Xi}(r\wedge\eta^{\Xi}_{N})-D\big(\bar{U}^{a,\Xi}(r\wedge\eta^{\Xi}_{N}-\tau)\big)|^{p}\Big]\nn\\
&~~~~+C\bigg(\triangle^{p(\frac{1}{2}-\alpha)}+\triangle^{-\frac{\alpha p}{2}}
\int_{0}^{t}\mathbb{E}\Big[\Big(\int^{r\wedge\eta^{\Xi}_{N}}_{\rho(r\wedge\eta^{\Xi}_{N})}\big|\sigma(U^{a,\Xi}(v),U^{a,\Xi}(v-\tau),\mathcal{L}_{v}^{U,\Xi})\big|^{2}\mathrm{d}v\Big)^{\frac{p}{4}}
\Big]\mathrm{d}r\bigg)\nn\\
&\leq\frac{1}{4}\mathbb{E}\!\Big[\sup_{0\leq r\leq t}\big|\bar{U}^{a,\Xi}(r\wedge\eta^{\Xi}_{N})\!-\!D\big(\bar{U}^{a,\Xi}(r\wedge\eta^{\Xi}_{N}\!-\!\tau)\big)\big|^{p}\Big]\nn\\
&~~~~+\!C\bigg(\!\triangle^{p(\frac{1}{2}-\alpha)}
\!+\!\triangle^{\frac{p}{2}(\frac{1}{2}-\alpha)}
\!\int_{0}^{t}\!\mathbb{E}\Big[1\!+\!|U^{a,\Xi}(r\wedge\eta^{\Xi}_{N})|^{\frac{p}{2}}\!+\!|U^{a,\Xi}(r\wedge\eta^{\Xi}_{N}\!-\!\tau)|^{\frac{p}{2}}\!+\!W^{\frac{p}{2}}_{2}(\mathcal{L}_{r\wedge\eta^{\Xi}_{N}}^{U,\Xi})
\!\Big]\mathrm{d}r\!\bigg)\nn\\
&\leq\frac{1}{4}\mathbb{E}\Big[\sup_{0\leq r\leq t}\big|\bar{U}^{a,\Xi}(r\wedge\eta^{\Xi}_{N})-D\big(\bar{U}^{a,\Xi}(r\wedge\eta^{\Xi}_{N}-\tau)\big)\big|^{p}\Big]\nn\\
&~~~~+C\bigg(\triangle^{p(\frac{1}{2}-\alpha)}+\int_{0}^{t}\mathbb{E}\Big[1+|U^{a,\Xi}(r\wedge\eta^{\Xi}_{N})|^{p}+|U^{a,\Xi}(r\wedge\eta^{\Xi}_{N}-\tau)|^{p}+W^{p}_{p}(\mathcal{L}_{r\wedge\eta^{\Xi}_{N}}^{U,\Xi})
\Big]\mathrm{d}r\bigg)\nn\\
&\leq\frac{1}{4}\mathbb{E}\Big[\sup_{0\leq r\leq t}\big|\bar{U}^{a,\Xi}(r\wedge\eta^{\Xi}_{N})-D\big(\bar{U}^{a,\Xi}(r\wedge\eta^{\Xi}_{N}-\tau)\big)\big|^{p}\Big]\nn\\
&~~~~+C\Big(\triangle^{p(\frac{1}{2}-\alpha)}
+\int_{0}^{t}\sup_{1\leq a\leq \Xi}\mathbb{E}\big[\sup_{0\leq v\leq r}|\bar{U}^{a,\Xi}(v\wedge\eta^{\Xi}_{N})|^{p}\big]\mathrm{d}r+1\Big)\nn\\
&\leq\frac{1}{4}\mathbb{E}\Big[\sup_{0\leq r\leq t}\big|\bar{U}^{a,\Xi}(r\wedge\eta^{\Xi}_{N})-D\big(\bar{U}^{a,\Xi}(r\wedge\eta^{\Xi}_{N}-\tau)\big)\big|^{p}\Big]\nn\\
&~~~~+C\Big(\int_{0}^{t}\sup_{1\leq a\leq \Xi}\mathbb{E}\big[\sup_{0\leq v\leq r}|\bar{U}^{a,\Xi}(v\wedge\eta^{\Xi}_{N})|^{p}\big]\mathrm{d}r+1\Big).
\end{align}
Using~$(\ref{eq3.3})$, $(\ref{eq4.7})$~and~Young's~inequality, we can deduce that
\begin{align}\label{eq4.13}
\tilde{J}^{a}_{2}&\leq C\mathbb{E}\bigg[\int_{0}^{t\wedge\eta^{\Xi}_{N}}\big|\bar{U}^{a,\Xi}(r)-D\big(\bar{U}^{a,\Xi}(r-\tau)\big)\big|^{p-2}\nn\\
&~~~~\times \Big(1+|U^{a,\Xi}(r)|^{2}+|U^{a,\Xi}(r-\tau)|^{2}+W^{2}_{2}(\mathcal{L}_{r}^{U,\Xi})\Big)\mathrm{d}r\bigg]\nn\\
&\leq C\mathbb{E}\Big[\int_{0}^{t\wedge\eta^{\Xi}_{N}}\big(|\bar{U}^{a,\Xi}(r)|^{p}+|\bar{U}^{a,\Xi}(r-\tau)|^{p}\big)\mathrm{d}r\Big]\nn\\
&~~~~+C\mathbb{E}\bigg[\int_{0}^{t\wedge\eta^{\Xi}_{N}}\Big(1+|U^{a,\Xi}(r)|^{p}+|U^{a,\Xi}(r-\tau)|^{p}
+W^{p}_{2}(\mathcal{L}_{r}^{U,\Xi})\Big) \mathrm{d}r\bigg]
\nn\\
&\leq C\Big(\int_{0}^{t}\sup_{1\leq a\leq \Xi}\mathbb{E}\big[\sup_{0\leq v\leq r}|\bar{U}^{a,\Xi}(v\wedge\eta^{\Xi}_{N})|^{p}\big]\mathrm{d}r+1\Big).
\end{align}
Applying the~BDG~inequality, Young's~inequality and~$\mathrm{H\ddot{o}lder}$'s~inequality we have
\begin{align}\label{eq4.14}
\tilde{J}^{a}_{3}
&\leq C\mathbb{E}\bigg[\Big(\int_{0}^{t\wedge\eta^{\Xi}_{N}}\big|\bar{U}^{a,\Xi}(r)-D\big(\bar{U}^{a,\Xi}(r-\tau)\big)\big|^{2p-2}
\big|\sigma(U^{a,\Xi}(r),U^{a,\Xi}(r-\tau),\mathcal{L}_{r}^{U,\Xi})\big|^{2}\mathrm{d}r\Big)^{\frac{1}{2}}\bigg]\nn\\
&\leq C\mathbb{E}\bigg[\sup_{0\leq r\leq t}\big|\bar{U}^{a,\Xi}(r\wedge\eta^{\Xi}_{N})-D\big(\bar{U}^{a,\Xi}(r\wedge\eta^{\Xi}_{N}-\tau)\big)\big|^{p-1}\nn\\
&~~~~~~~~\times\Big(\int_{0}^{t\wedge\eta^{\Xi}_{N}}
\big|\sigma(U^{a,\Xi}(r),U^{a,\Xi}(r-\tau),\mathcal{L}_{r}^{U,\Xi})\big|^{2}\mathrm{d}r\Big)^{\frac{1}{2}}\bigg]\nn\\
&\leq\frac{1}{4}\mathbb{E}\Big[\sup_{0\leq r\leq t}\big|\bar{U}^{a,\Xi}(r\wedge\eta^{\Xi}_{N})-D\big(\bar{U}^{a,\Xi}(r\wedge\eta^{\Xi}_{N}-\tau)\big)\big|^{p}\Big]\nn\\
&~~~~+C\mathbb{E}\Big[\int_{0}^{t\wedge\eta^{\Xi}_{N}}\big|\sigma(U^{a,\Xi}(r),U^{a,\Xi}(r-\tau),\mathcal{L}_{r}^{U,\Xi})\big|^{p}\mathrm{d}r\Big]\nn\\
&\leq\frac{1}{4}\mathbb{E}\Big[\sup_{0\leq r\leq t}\big|\bar{U}^{a,\Xi}(r\wedge\eta^{\Xi}_{N})-D\big(\bar{U}^{a,M}(r\wedge\eta^{\Xi}_{N}-\tau)\big)\big|^{p}\Big]\nn\\
&~~~~+C\mathbb{E}\bigg[\int_{0}^{t\wedge\eta^{\Xi}_{N}}
\Big(1+|U^{a,\Xi}(r)|^{p}+|U^{a,\Xi}(r-\tau)|^{p}+W^{p}_{p}(\mathcal{L}_{r}^{U,\Xi})\Big)\mathrm{d}r\bigg]\nn\\
&\leq\frac{1}{4}\mathbb{E}\Big[\sup_{0\leq r\leq t}\big|\bar{U}^{a,\Xi}(r\wedge\eta^{\Xi}_{N})-D\big(\bar{U}^{a,\Xi}(r\wedge\eta^{\Xi}_{N}-\tau)\big)\big|^{p}\Big]\nn\\
&~~~~+C\Big(\int_{0}^{t}\sup_{1\leq a\leq \Xi}\mathbb{E}\big[\sup_{0\leq v\leq r}|\bar{U}^{a,\Xi}(v\wedge\eta^{\Xi}_{N})|^{p}\big]\mathrm{d}r+1\Big).
\end{align}
Inserting~$(\ref{eq4.12})$-$(\ref{eq4.14})$~into~$(\ref{eq4.11})$, we derive that
\begin{align*}
&\mathbb{E}\Big[\sup_{0\leq r\leq t}\big|\bar{U}^{a,\Xi}(r\wedge\eta^{\Xi}_{N})-D\big(\bar{U}^{a,\Xi}(r\wedge\eta^{\Xi}_{N}-\tau)\big)\big|^{p}\Big]\nn\\
&\leq C\int_{0}^{t}\sup_{1\leq a\leq \Xi}\mathbb{E}\big[\sup_{0\leq v\leq r}|\bar{U}^{a,\Xi}(v\wedge\eta^{\Xi}_{N})|^{p}\big]\mathrm{d}r+C.
\end{align*}
Combining the above inequality with~$(\ref{eq4.9})$~and using the~Gronwall~inequality arrive at
\begin{align*}
\sup_{1\leq a\leq M}\mathbb{E}\Big[\sup_{0\leq r\leq t}\big|\bar{U}^{a,\Xi}(r\wedge\eta^{\Xi}_{N})\big|^{p}\Big]\leq C.
\end{align*}
Therefore, letting $N\rightarrow \infty$, using Fatou's lemma one can see that
\begin{align*}
\sup_{1\leq a\leq \Xi}\mathbb{E}\big[\sup_{0\leq t\leq T}|\bar{U}^{a,\Xi}(t)|^{p}\big]\leq C.
\end{align*}
\end{proof}

\subsection{Convergence Rate}\la{s6}
We investigate in this section the convergence rate between the numerical solution and the exact solution.
\begin{lemma}\label{le4.3}
Under  Assumption $\ref{a1}$,
for any $T>0$,
\begin{align}\label{eq4.15}
\sup_{1\leq a\leq \Xi}\sup_{0\leq n\leq \lfloor T/\triangle\rfloor}\mathbb{E}\big[\sup_{t_{n}\leq t<t_{n+1}}|\bar{U}^{a,\Xi}(t)-U^{a,\Xi}_{n}|^{p}\big]\leq C\triangle^{\frac{p}{2}},~~t\in[0,T]. \end{align}

\end{lemma}
\begin{proof}
Let~$\tilde{N}=\lfloor T/\triangle\rfloor$. For each~$~a=1,\cdot\cdot\cdot,\Xi$ and $t\in[0,T]$, there exists~$n,~0\leq n\leq\tilde{N}$~such that~$t\in[t_{n},t_{n+1})$. Then one observes
\begin{align}\label{eq4.16}
&\mathbb{E}\Big[\sup_{t_{n}\leq t<t_{n+1}}\big|\bar{U}^{a,\Xi}(t)-D\big(\bar{U}^{a,\Xi}(t-\tau)\big)-U^{a,\Xi}_{n}+D(U^{a,\Xi}_{n-n_{0}})\big|^{p}\Big]\nn\\
&\leq C\mathbb{E}\bigg[\sup_{t_{n}\leq t<t_{n+1}}\Big(\big|b_{\triangle}(U^{a,\Xi}_{n},U^{a,\Xi}_{n-n_{0}},\mathcal{L}^{U_{n},\Xi})(t-t_{n})\big|^{p}\nn\\
&~~~~~~~~~~~~~~~~+\big|\sigma(U^{a,\Xi}_{n},U^{a,\Xi}_{n-n_{0}},\mathcal{L}^{U_{n},\Xi})(B^{a}(t)-B^{a}(t_{n}))\big|^{p}\Big)
\bigg]\nn\\
&\leq C\bigg(\triangle^{(1-\alpha)p}+\mathbb{E}\Big[\sup_{t_{n}\leq t<t_{n+1}}\big|\sigma(U^{a,\Xi}_{n},U^{a,\Xi}_{n-n_{0}},\mathcal{L}^{U_{n},\Xi})(B^{a}(t)-B^{a}(t_{n}))\big|^{p}\Big]\bigg).
\end{align}
By~Doob's martingale inequality and Lemma~$\ref{le4.2}$, we derive that
\begin{align}\label{eq4.17}
&\mathbb{E}\Big[\sup_{t_{n}\leq t<t_{n+1}}\big|\sigma(U^{a,\Xi}_{n},U^{a,\Xi}_{n-n_{0}},\mathcal{L}^{U_{n},\Xi})(B^{a}(t)-B^{a}(t_{n}))\big|^{p}\Big]\nn\\
&\leq C\triangle^{\frac{p}{2}}\mathbb{E}\Big[1+|U^{a,\Xi}_{n}|^{p}+|U^{a,\Xi}_{n-n_{0}}|^{p}+W^{p}_{p}(\mathcal{L}^{U_n,\Xi})\Big]\nn\\
&\leq C\triangle^{\frac{p}{2}}.
\end{align}
Therefore, substituting~$(\ref{eq4.17})$~into~$(\ref{eq4.16})$~yields that
\begin{align}\label{eq4.18}
\sup_{0\leq n\leq \tilde{N}}\mathbb{E}\Big[\sup_{t_{n}\leq t<t_{n+1}}\big|\bar{U}^{a,\Xi}(t)-D\big(\bar{U}^{a,\Xi}(t-\tau)\big)-U^{a,\Xi}_{n}+D(U^{a,\Xi}_{n-n_{0}})\big|^{p}\Big]\leq C\triangle^{\frac{p}{2}}.
\end{align}
Next, by Lemma~$\ref{l3.4}$, we have
\begin{align*}
 |\bar{U}^{a,\Xi}(t)-U^{a,\Xi}_{n}|^{p}
&\leq \frac{1}{(1-\lambda)^{p-1}} \big|\bar{U}^{a,\Xi}(t)-D\big(\bar{U}^{a,\Xi}(t-\tau)\big)-U^{a,\Xi}_{n}+D(U^{a,\Xi}_{n-n_{0}})\big|^{p}\nn\\
&~~~~+\frac{1}{ \lambda^{p-1}} \big|D\big(\bar{U}^{a,\Xi}(t-\tau)\big)-D(U^{a,\Xi}_{n-n_{0}})\big|^{p},
\end{align*}
which together with Assumption \ref{a1} implies
\begin{align*}
&\mathbb{E}\Big[\sup_{t_{n}\leq t<t_{n+1}}\big|\bar{U}^{a,\Xi}(t)-U^{a,\Xi}_{n}\big|^{p}\Big]\nn\\
&\leq\frac{1}{(1-\lambda)^{p-1}}\mathbb{E}\Big[\sup_{t_{n}\leq t<t_{n+1}}\big|\bar{U}^{a,\Xi}(t)-D\big(\bar{U}^{a,\Xi}(t-\tau)\big)-U^{a,\Xi}_{n}+D(U^{a,\Xi}_{n-n_{0}})\big|^{p}\Big]\nn\\
&~~~~+\lambda\mathbb{E}\big[\sup_{t_{n}\leq t<t_{n+1}}|\bar{U}^{a,\Xi}(t-\tau)-U^{a,\Xi}_{n-n_{0}}|^{p}\big]\nn\\
&\leq\frac{1}{(1-\lambda)^{p-1}}\mathbb{E}\Big[\sup_{t_{n}\leq t<t_{n+1}}\big|\bar{U}^{a,\Xi}(t)-D\big(\bar{U}^{a,\Xi}(t-\tau)\big)-U^{a,\Xi}_{n}+D(U^{a,\Xi}_{n-n_{0}})\big|^{p}\Big]\nn\\
&~~~~+\lambda\mathbb{E}\big[\sup_{t_{n-n_{0}}\leq t<t_{n-n_{0}+1}}|\bar{U}^{a,\Xi}(t)-U^{a,\Xi}_{n-n_{0}}|^{p}\big].
\end{align*}
According to \eqref{eq4.18} we calculate
\begin{align*}
&\sup_{0 \leq n\leq \tilde{N}}\mathbb{E}\Big[\sup_{t_{n}\leq t<t_{n+1}}|\bar{U}^{a,\Xi}(t)-U^{a,\Xi}_{n}|^{p}\Big]\nn\\
&\leq\frac{1}{(1-\lambda)^{p-1}}\sup_{0 \leq n\leq  \tilde{N}}\mathbb{E}\Big[\sup_{t_{n}\leq t<t_{n+1}}\big|\bar{U}^{a,\Xi}(t)-D\big(\bar{U}^{a,\Xi}(t-\tau)\big)-U^{a,\Xi}_{n}+D(U^{a,\Xi}_{n-n_{0}})\big|^{p}\Big]\nn\\
&~~~~+\lambda\sup_{- n_0\leq n\leq  \tilde{N} -n_0 }\mathbb{E}\Big[\sup_{t_{n}\leq t<t_{n+1}}\big|\bar{U}^{a,\Xi}(t)-U^{a,\Xi}_{n}\big|^{p}\Big]\nn\\
&\leq\frac{C}{(1-\lambda)^{p-1}} \triangle^{\frac{p}{2}}+ \lambda\sup_{- n_0\leq n< 0 }\mathbb{E}\Big[\sup_{t_{n}\leq t<t_{n+1}}\big|\bar{U}^{a,\Xi}(t)-U^{a,\Xi}_{n}\big|^{p}\Big]\nn\\
&~~~~+\lambda\sup_{0\leq n\leq\tilde{N}}\mathbb{E}\Big[\sup_{t_{n}\leq t<t_{n+1}}\big|\bar{U}^{a,\Xi}(t)-U^{a,\Xi}_{n}\big|^{p}\Big]\nn\\
&\leq\frac{C}{(1-\lambda)^{p }} \triangle^{\frac{p}{2}}
+\frac{\lambda}{ 1-\lambda}\sup_{- n_0\leq n< 0}\mathbb{E}\Big[\sup_{t_{n}\leq t<t_{n+1}}\big|\bar{U}^{a,\Xi}(t)-U^{a,\Xi}_{n}\big|^{p}\Big].
\end{align*}
This yields that
\begin{align}\label{eq4.19}
&\sup_{0 \leq n\leq \tilde{N}}\mathbb{E}\Big[\sup_{t_{n}\leq t<t_{n+1}}|\bar{U}^{a,\Xi}(t)-U^{a,\Xi}_{n}|^{p}\Big]\nn\\
&\leq\frac{C}{(1-\lambda)^{p }} \triangle^{\frac{p}{2}}
+\frac{\lambda}{ 1-\lambda}\sup_{- n_0\leq n< 0}\mathbb{E}\Big[\sup_{t_{n}\leq t<t_{n+1}}\big|\bar{U}^{a,\Xi}(t)-U^{a,\Xi}_{n}\big|^{p}\Big].
\end{align}
From \eqref{eq3.2} we can infer that
\begin{align}\label{eq4.20}
\sup_{- n_0\leq n<0}\!\mathbb{E}\Big[\sup_{t_{n}\leq t<t_{n+1}}\big|\!\bar{U}^{a,\Xi}(t)\!-\!U^{a,\Xi}_{n}\!\big|^{p}\!\Big]
\leq\sup_{- n_0\leq n< 0 }\!\mathbb{E}\!\Big[\sup_{t_{n}\leq t<t_{n+1}}\big|\xi^{a}(t)\!-\!\xi^{a}(t_{n})\big|^{p}\!\Big]\leq K_{0}\triangle^{\frac{p}{2}}.
\end{align}
Inserting \eqref{eq4.20} into (\ref{eq4.19}) yields
\begin{align*}
 \sup_{0 \leq n\leq\tilde{N}}\mathbb{E}\Big[\sup_{t_{n}\leq t<t_{n+1}}\big|\bar{U}^{a,\Xi}(t)-U^{a,\Xi}_{n}\big|^{p}\Big]\leq\Big(\frac{C}{(1-\lambda)^{p }} +\frac{\lambda K_0}{ 1-\lambda }\Big)\triangle^{\frac{p}{2}}.
\end{align*}
The required result $(\ref{eq4.15})$~follows.
\end{proof}

\begin{lemma}\label{le4.4}
Under Assumptions $\ref{a1}$~and~$\ref{a2}$ with~$p\geq4(c+1)$,  for any~$q\in[2,\frac{p}{2(c+1)}]$,
$$
\sup_{1\leq a\leq \Xi}\mathbb{E}\big[\sup_{0\leq t\leq T}|S^{a,\Xi}(t)-\bar{U}^{a,\Xi}(t)|^{q}\big]\leq C\triangle^{\alpha q},~~~T\geq0. $$
\end{lemma}
\begin{proof}
For each~$a=1,\cdot\cdot\cdot,\Xi$, $t\in[0,T]$, define
$$
\Gamma^{a,\Xi}(t)=S^{a,\Xi}(t)-\bar{U}^{a,\Xi}(t)-D(S^{a,\Xi}(t-\tau))+D\big(\bar{U}^{a,\Xi}(t-\tau)\big).
$$
For the fixed $q\in[2,\frac{p}{2(c+1)}]$, due to Lemma \ref{l3.4} and Assumption \ref{a1}, we can derive that
\begin{align*}
\sup_{0\leq r\leq t}\big|S^{a,\Xi}(r)\!-\!\bar{U}^{a,\Xi}(r)\big|^{q}&\leq\lambda\sup_{0\leq r\leq t}\big|S^{a,\Xi}(r\!-\!\tau)\!-\!\bar{U}^{a,\Xi}(r\!-\!\tau)\big|^{q}\!+\!\frac{1}{(1-\lambda)^{q-1}}\sup_{0\leq r\leq t}|\Gamma^{a,\Xi}(r)|^{q}\nn\\
&\leq\lambda\sup_{-\tau\leq r\leq0}\big|S^{a,\Xi}(r)-\bar{U}^{a,\Xi}(r)\big|^{q}+\lambda\sup_{0\leq r\leq t}\big|S^{a,\Xi}(r)-\bar{U}^{a,\Xi}(r)\big|^{q}\nn\\
&~~~~+\frac{1}{(1-\lambda)^{q-1}}\sup_{0\leq r\leq t}|\Gamma^{a,\Xi}(r)|^{q},
\end{align*}
which implies
\begin{align}\label{eq4.21}
\mathbb{E}\Big[\sup_{0\leq r\leq t}\big|S^{a,\Xi}(r)-\bar{U}^{a,\Xi}(r)\big|^{q}\Big]\leq\frac{1}{(1-\lambda)^{q}}\mathbb{E}\big[\sup_{0\leq r\leq t}|\Gamma^{a,\Xi}(r)|^{q}\big].
\end{align}
Using the~It$\mathrm{\hat{o}}$~formula yields
\begin{align*}
|\Gamma^{a,\Xi}(t)|^{q}
&\leq q\int_{0}^{t}|\Gamma^{a,\Xi}(r)|^{q-2}\big(\Gamma^{a,\Xi}(r)\big)^{T}\Big\{b(S^{a,\Xi}(r),S^{a,\Xi}(r-\tau),\mathcal{L}_{r}^{S,\Xi})\\
&~~~~~~~~~~~~~~~~~-b_{\triangle}(U^{a,\Xi}(r),U^{a,\Xi}(r-\tau),\mathcal{L}_{r}^{U,\Xi})\Big\}\mathrm{d}r\\
&~~~~+\frac{q(q-1)}{2}\int_{0}^{t}|\Gamma^{a,\Xi}(r)|^{q-2}\Big|\sigma(S^{a,\Xi}(r),S^{a,\Xi}(r-\tau),\mathcal{L}_{r}^{S,\Xi})\\
&~~~~~~~~~~~~~~~~~-\sigma(U^{a,\Xi}(r),U^{a,\Xi}(r-\tau),\mathcal{L}_{r}^{U,\Xi})\Big|^{2}\mathrm{d}r\\
&~~~~+q\int_{0}^{t}|\Gamma^{a,\Xi}(r)|^{q-2}\big(\Gamma^{a,\Xi}(r)\big)^{T} \Big\{\sigma(S^{a,\Xi}(r),S^{a,\Xi}(r-\tau),\mathcal{L}_{r}^{S,\Xi})\\
&~~~~~~~~~~~~~~~~~-\sigma(U^{a,\Xi}(r),U^{a,\Xi}(r-\tau),\mathcal{L}_{r}^{U,\Xi})\Big\}\mathrm{d}B^{a}(r).
\end{align*}
This implies
\begin{align}\label{eq4.22}
\mathbb{E}\big[\sup_{0\leq r\leq t}|\Gamma^{a,\Xi}(r)|^{q}\big]\leq \sum_{l=1}^5 Q^{a,\Xi}_{l},
\end{align}
where
\begin{align*}
Q^{a,\Xi}_{1}&=q\mathbb{E}\bigg[\sup_{0\leq r\leq t}\int_{0}^{r}|\Gamma^{a,\Xi}(v)|^{q-2}\big(\Gamma^{a,\Xi}(v)\big)^{T}\Big\{b(S^{a,\Xi}(v),S^{a,\Xi}(v-\tau),\mathcal{L}_{v}^{S,\Xi})
\\
&~~~~~~~~~~~~~~~~~~~~~~~~-b(\bar{U}^{a,\Xi}(v),\bar{U}^{a,\Xi}(v-\tau),\mathcal{L}_{v}^{U,\Xi})
\Big\}\mathrm{d}v\bigg],\\
Q^{a,\Xi}_{2}&=q\mathbb{E}\bigg[\sup_{0\leq r\leq t}\int_{0}^{r}|\Gamma^{a,\Xi}(v)|^{q-2}\big(\Gamma^{a,\Xi}(v)\big)^{T}\Big\{b(\bar{U}^{a,\Xi}(v),\bar{U}^{a,\Xi}(v-\tau),\mathcal{L}_{v}^{U,\Xi})\\
&~~~~~~~~~~~~~~~~~~~~~~~~-b(U^{a,\Xi}(v),U^{a,\Xi}(v-\tau),\mathcal{L}_{v}^{U,\Xi})\Big\}\mathrm{d}v\bigg],\\
Q^{a,\Xi}_{3}&=q\mathbb{E}\bigg[\sup_{0\leq r\leq t}\int_{0}^{r}|\Gamma^{a,\Xi}(v)|^{q-2}(\Gamma^{a,\Xi}(v))^{T}
 \Big\{b(U^{a,\Xi}(v),U^{a,\Xi}(v-\tau),\mathcal{L}_{v}^{U,\Xi})\\
 &~~~~~~~~~~~~~~~~~~~~~~~~-b_{\triangle}(U^{a,\Xi}(v),U^{a,\Xi}(v-\tau),\mathcal{L}_{v}^{U,\Xi})\Big\}\mathrm{d}v\bigg],\\
Q^{a,\Xi}_{4}&=\frac{q(q-1)}{2}\mathbb{E}\bigg[\sup_{0\leq r\leq t} \int_{0}^{r}|\Gamma^{a,\Xi}(v)|^{q-2}
  \Big|\sigma(S^{a,\Xi}(v),S^{a,\Xi}(v-\tau),\mathcal{L}_{v}^{S,\Xi})\\
  &~~~~~~~~~~~~~~~~~~~~~~~~-
  \sigma(U^{a,\Xi}(v),U^{a,\Xi}(v-\tau),\mathcal{L}_{v}^{U,\Xi})\Big|^{2}\mathrm{d}v\bigg]
  \\
\end{align*}
and
\begin{align*}
Q^{a,\Xi}_{5}&=q\mathbb{E}\bigg[\sup_{0\leq r\leq t}\int_{0}^{r} |\Gamma^{a,\Xi}(v)|^{q-2}(\Gamma^{a,\Xi}(v))^{T} \Big\{\sigma(S^{a,\Xi}(v),S^{a,\Xi}(v-\tau),\mathcal{L}_{v}^{S,\Xi})\\
&~~~~~~~~~~~~~~~~~~~~~~~~-\sigma(U^{a,\Xi}(v),U^{a,\Xi}(v-\tau),\mathcal{L}_{v}^{U,\Xi})\Big\}\mathrm{d}B^{a}(v)\bigg].
 \end{align*}
Then, according to Young's inequality and Lemma~$\ref{le4.3}$, we have
\begin{align}\label{eq4.23}
Q^{a,\Xi}_{1}
&\leq C\mathbb{E}\bigg[\int^{t}_{0}\Big(\big|S^{a,\Xi}(r)-\bar{U}^{a,\Xi}(r)\big|^{q-2}+\big|D(S^{a,\Xi}(r-\tau))
-D\big(\bar{U}^{a,\Xi}(r-\tau)\big)\big|^{q-2}\Big)\nn\\
&~~~\times\Big(\big|S^{a,\Xi}(r)-\bar{U}^{a,\Xi}(r)\big|^{2}+\big|S^{a,\Xi}(r-\tau)-\bar{U}^{a,\Xi}(r-\tau)\big|^{2}
+\mathbb{W}^{2}_{2}(\mathcal{L}_{r}^{S,\Xi},\mathcal{L}_{r}^{U,\Xi})\Big)\mathrm{d}r\bigg]\nn\\
&\leq C\mathbb{E}\bigg[\!\int^{t}_{0}\Big(\big|S^{a,\Xi}(r)\!-\!\bar{U}^{a,\Xi}(r)\big|^{q}
\!+\!\big|S^{a,\Xi}(r\!-\!\tau)\!-\!\bar{U}^{a,\Xi}(r-\tau)\big|^{q}\!+\!\mathbb{W}^{q}_{q}(\mathcal{L}_{r}^{S,\Xi},\!\mathcal{L}_{r}^{U,\Xi})\Big)\mathrm{d}r\!\bigg]\nn\\
&\leq C \mathbb{E}\bigg[\int^{t}_{0}\big|S^{a,\Xi}(r)-\bar{U}^{a,\Xi}(r)\big|^{q}\mathrm{d}r
+\int^{t}_{0}\frac{1}{\Xi}\sum^{\Xi}_{j=1}\big|S^{j,\Xi}(r)-U^{j,\Xi}(r)\big|^{q}\mathrm{d}r\bigg] \nn\\
&\leq C\mathbb{E}\bigg[\int^{t}_{0}\big|S^{a,\Xi}(r)-\bar{U}^{a,\Xi}(r)\big|^{q}\mathrm{d}r
+ \int^{t}_{0}\frac{1}{\Xi}\sum^{\Xi}_{j=1}\big|S^{j,\Xi}(r)-\bar{U}^{j,\Xi}(r)\big|^{q}\mathrm{d}r \nn\\
&~~~~~~~~~~~~~+\int^{t}_{0}\frac{1}{\Xi}\sum^{\Xi}_{j=1}\big|\bar{U}^{j,\Xi}(r)-U^{j,\Xi}(r)\big|^{q}\mathrm{d}r\bigg] \nn\\
&\leq C \int^{t}_{0}\sup_{1\leq a\leq \Xi}\mathbb{E}\Big[\sup_{0\leq v\leq r}\big|S^{a,\Xi}(v)-\bar{U}^{a,\Xi}(v)\big|^{q}\Big]\mathrm{d}r+C\triangle^{\frac{q}{2}}.
\end{align}
By \eqref{eq3.2}, Assumption~$\ref{a2}$~and Lemma~$\ref{le4.2}$, we can infer that
\begin{align}\label{eq4.24}
Q^{a,\Xi}_{2}&\leq C\mathbb{E}\bigg[\int_{0}^{t}|\Gamma^{a,\Xi}(r)|^{q-1}
\Big(1+|\bar{U}^{a,\Xi}(r)|^{c}+|\bar{U}^{a,\Xi}(r-\tau)|^{c}+\!|U^{a,\Xi}(r)|^{c}+\!|U^{a,\Xi}(r-\tau)|^{c}\Big)\nn\\
&~~~~~~~~~~~~~~~~~~~~~~\times\Big(\big|\bar{U}^{a,\Xi}(r)-U^{a,\Xi}(r)\big|+\big|\bar{U}^{a,\Xi}(r-\tau)
-U^{a,\Xi}(r-\tau)\big|\Big)\mathrm{d}r\bigg]\nn\\
&\leq C\mathbb{E}\bigg[\!\sup_{0\leq r\leq t}\!|\Gamma^{a,\Xi}(r)|^{q-1}\!\int_{0}^{t}\!\Big(1\!+\!|\bar{U}^{a,\Xi}(r)|^{c}\!+\!|\bar{U}^{a,\Xi}(r\!-\!\tau)|^{c}\!+\!|U^{a,\Xi}(r)|^{c}\!+\!|U^{a,\Xi}(r\!-\!\tau)|^{c}\Big)\nn\\
&~~~~~~~~~~~~~~~~~~~~~~~\times\Big(\big|\bar{U}^{a,\Xi}(r)-U^{a,\Xi}(r)\big|+\big|\bar{U}^{a,\Xi}(r-\tau)-U^{a,\Xi}(r-\tau)\big|\Big)\mathrm{d}r\bigg]\nn\\
&\leq \frac{1}{4}\mathbb{E}\big[\sup_{0\leq r\leq t}|\Gamma^{a,\Xi}(r)|^{q}\big]\nn\\
&~~~~+ C\mathbb{E}\bigg[\Big\{\int_{0}^{t}\Big(1+|\bar{U}^{a,\Xi}(r)|^{c}+|\bar{U}^{a,\Xi}(r-\tau)|^{c}
+|U^{a,\Xi}(r)|^{c}+|U^{a,\Xi}(r-\tau)|^{c}\Big)\nn\\
&~~~~~~~~~~~~~~~~~~~\times\Big(\big|\bar{U}^{a,\Xi}(r)-U^{a,\Xi}(r)\big|+\big|\bar{U}^{a,\Xi}(r-\tau)-U^{a,\Xi}(r-\tau)\big|\Big)\mathrm{d}s\Big\}^{q}\bigg]\nn\\
&\leq \frac{1}{4}\mathbb{E}\big[\sup_{0\leq r\leq t}|\Gamma^{a,\Xi}(r)|^{q}\big] \nn\\
&~~~~+C\int_{0}^{t}\bigg(\mathbb{E}\Big[1+|\bar{U}^{a,\Xi}(r)|^{p}+|\bar{U}^{a,\Xi}(r-\tau)|^{p}+|U^{a,\Xi}(r)|^{p}+|U^{a,\Xi}(r-\tau)|^{p}\Big]\bigg)^{\frac{cq}{p}}\nn\\
&~~~~~~~~~~~~~~~~~~ \times\bigg(\mathbb{E}\bigg[\!\Big(\big|\bar{U}^{a,\Xi}(r)\!-\!U^{a,\Xi}(r)\big|
\!+\!\big|\bar{U}^{a,\Xi}(r\!-\!\tau)\!-\!U^{a,\Xi}(r\!-\!\tau)\big|\Big)^{\frac{pq}{p-cq}}\bigg]\bigg)^{\frac{p-cq}{p}}\mathrm{d}r\nn\\
&\leq \frac{1}{4}\mathbb{E}\big[\sup_{0\leq r\leq t}|\Gamma^{a,\Xi}(r)|^{q}\big] \nn\\
&~+\!C\bigg\{\!\int_{0}^{t}\!\Big(\!\mathbb{E}\Big[\!\big|\bar{U}^{a,\Xi}(r)\!-\!U^{a,\Xi}(r)\!\big|^{\frac{pq}{p-cq}}\Big]\!\Big)^{\frac{p-cq}{p}}\!\mathrm{d}r\!+\!\sup_{-\tau\leq r\leq0}\!\Big(\!\mathbb{E}\Big[\!\big|\bar{U}^{a,\Xi}(r)\!-\!U^{a,\Xi}(r)\big|^{\frac{pq}{p-cq}}\!\Big]\Big)^{\frac{p-cq}{p}}\!\bigg\}\nn\\
&\leq\frac{1}{4}\mathbb{E}\Big[\sup_{0\leq r\leq t}|\Gamma^{a,\Xi}(r)|^{q}\Big]+C\triangle^{\frac{q}{2}}.
\end{align}
One notices from~$(\ref{eq3.6})$~and~$(\ref{eq4.1})$~that
\begin{align*}
&\mathbb{E}\Big[\int_{0}^{t}\big|b(U^{a,\Xi}(r),U^{a,\Xi}(r-\tau),\mathcal{L}_{r}^{U,\Xi})
-b_{\triangle}(U^{a,\Xi}(r),U^{a,\Xi}(r-\tau),\mathcal{L}_{r}^{U,\Xi})\big|^{q}\mathrm{d}r\Big]\nn\\
&\leq\triangle^{\alpha q}\int_{0}^{t}
\mathbb{E}\bigg[\frac{|b(U^{a,\Xi}(r),U^{a,\Xi}(r-\tau),\mathcal{L}_{r}^{U,\Xi})|^{2q}}{\Big(1+\triangle^{\alpha}|b(U^{a,\Xi}(r),U^{a,\Xi}(r-\tau),\mathcal{L}_{r}^{U,\Xi})|\Big)^{q}}\bigg]\mathrm{d}r\nn\\
&\leq C\triangle^{\alpha q}\int_{0}^{t}
\mathbb{E}\bigg[\Big(1+|U^{a,\Xi}(r)|^{c+1}+|U^{a,\Xi}(r-\tau)|^{c+1}+W_{2}(\mathcal{L}_{r}^{U,\Xi})\Big)^{2q}\bigg]\mathrm{d}r\nn\\
&\leq C\triangle^{\alpha q}.
\end{align*}
By Young's inequality and the above inequality, we obtain that
\begin{align}\label{eq4.25}
Q^{a,\Xi}_{3}&\leq C\mathbb{E}\Big[\sup_{0\leq r\leq t}|\Gamma^{a,\Xi}(r)|^{q-1}
\int_{0}^{t}\big|b(U^{a,\Xi}(r),U^{a,\Xi}(r-\tau),\mathcal{L}_{r}^{U,\Xi})\nn\\
&~~~~~~~~~~~~~~~~~~~~~~~~~~~~~~~~~~~~-b_{\triangle}(U^{a,\Xi}(r),U^{a,\Xi}(r-\tau),\mathcal{L}_{r}^{U,\Xi})
\big|\mathrm{d}r\Big]\nn\\
&\leq\frac{1}{4}\mathbb{E}\big[\sup_{0\leq r\leq t}|\Gamma^{a,\Xi}(r)|^{q}\big]\nn\\
&~~~~+C\mathbb{E}\Big[\int_{0}^{t}\big|b(U^{a,\Xi}(r),U^{a,\Xi}(r-\tau),\mathcal{L}_{r}^{U,\Xi})
-b_{\triangle}(U^{a,\Xi}(r),U^{a,\Xi}(r-\tau),\mathcal{L}_{r}^{U,\Xi})\big|^{q}\mathrm{d}r\Big]\nn\\
&\leq\frac{1}{4}\mathbb{E}\big[\sup_{0\leq r\leq t}|\Gamma^{a,\Xi}(r)|^{q}\big] + C\triangle^{\alpha q}.
\end{align}
Due to Assumption~$\ref{a1}$, $(\ref{eq4.23})$ and~$(\ref{eq3.2})$, by applying Young's inequality and~$\mathrm{H\ddot{o}lder}$'s~inequality, we derive that
\begin{align}\label{eq4.26}
Q^{a,\Xi}_{4}&\leq C\mathbb{E}\bigg[\int^{t}_{0}\Big(\big|S^{a,\Xi}(r)-\bar{U}^{a,\Xi}(r)\big|^{q-2}+\big|D(S^{a,\Xi}(r-\tau))-D\big(\bar{U}^{a,\Xi}(r-\tau)\big)\big|^{q-2}\Big)\nn\\
&~~~~~~~~\times\Big(\big|S^{a,\Xi}(r)-U^{a,\Xi}(r)\big|^{2}\!+\!\big|S^{a,\Xi}(r\!-\tau)-U^{a,\Xi}(r\!-\tau)\big|^{2}
\!+\mathbb{W}^{2}_{2}(\mathcal{L}_{r}^{S,\Xi},\mathcal{L}_{r}^{U,\Xi})\Big)\mathrm{d}r\bigg]\nn\\
&\leq C\mathbb{E}\bigg[\int^{t}_{0}\Big(\big|S^{a,\Xi}(r)-\bar{U}^{a,\Xi}(r)\big|^{q}+\big|S^{a,\Xi}(r-\tau)-\bar{U}^{a,\Xi}(r-\tau)\big|^{q}\nn\\
&~~~~+\big|S^{a,\Xi}(r)-U^{a,\Xi}(r)\big|^{q}+\big|S^{a,\Xi}(r-\tau)-U^{a,\Xi}(r-\tau)\big|^{q}
+\mathbb{W}^{q}_{2}(\mathcal{L}_{r}^{S,\Xi},\mathcal{L}_{r}^{U,\Xi})\Big)\mathrm{d}r\bigg]\nn\\
&\leq C\mathbb{E}\bigg[\int^{t}_{0}\Big(\big|S^{a,\Xi}(r)-\bar{U}^{a,\Xi}(r)\big|^{q}+\big|S^{a,\Xi}(r-\tau)-\bar{U}^{a,\Xi}(r-\tau)\big|^{q}\nn\\
&~~~~+\big|\bar{U}^{a,\Xi}(r)-U^{a,\Xi}(r)\big|^{q}+\big|\bar{U}^{a,\Xi}(r-\tau)-U^{a,\Xi}(r-\tau)\big|^{q}
+\mathbb{W}^{q}_{q}(\mathcal{L}_{r}^{S,\Xi},\mathcal{L}_{r}^{U,\Xi})\Big)\mathrm{d}r\bigg]\nn\\
&\leq C\mathbb{E}\bigg[\int^{t}_{0}\Big(\big|S^{a,\Xi}(r)\!-\!\bar{U}^{a,\Xi}(r)\big|^{q}\!+\!\big|S^{a,\Xi}(r\!-\!\tau)\!-\bar{U}^{a,\Xi}(r-\!\tau)\big|^{q}
\!+\!\mathbb{W}^{q}_{q}(\mathcal{L}_{r}^{S,\Xi},\mathcal{L}_{r}^{U,\Xi})\!\Big)\mathrm{d}r\!\bigg]\nn\\
&~~~~+C\mathbb{E}\Big[\int^{t}_{0}\big|\bar{U}^{a,\Xi}(r)\!-\!U^{a,\Xi}(r)\big|^{q}\mathrm{d}r\Big]
+C\sup_{-\tau\leq r\leq0}\mathbb{E}\Big[\big|\bar{U}^{a,\Xi}(r)-U^{a,\Xi}(r)\big|^{q}\Big] \nn\\
&\leq C \int^{t}_{0}\sup_{1\leq a\leq \Xi}\mathbb{E}\Big[\sup_{0\leq v\leq r}\big|S^{a,\Xi}(v)-\bar{U}^{a,\Xi}(v)\big|^{q}\Big]\mathrm{d}r+C\triangle^{\frac{q}{2}}.
\end{align}
By the~BDG~inequality, we obtain
\begin{align}\label{eq4.27}
Q^{a,\Xi}_{5}
&\leq C\mathbb{E}\bigg[\!\Big(\int_{0}^{t}\!|\Gamma^{a,\Xi}(r)|^{2q-2}\nn\\
&~~~~~~~~~\times\big|\sigma(S^{a,\Xi}(r),\!S^{a,\Xi}(r\!-\!\tau),\mathcal{L}_{r}^{S,\Xi})\!-\!\sigma(U^{a,\Xi}(r),\!U^{a,\Xi}(r\!-\!\tau),\mathcal{L}_{r}^{U,\Xi})\big|^{2}
\mathrm{d}s\!\Big)^{\frac{1}{2}}\!\bigg]\nn\\
&\leq C\mathbb{E}\bigg[\sup_{0\leq r\leq t}|\Gamma^{a,\Xi}(r)|^{q-1}\nn\\
&~~~~~~~~~\times\Big(\int_{0}^{t}\big|\sigma(S^{a,\Xi}(r),S^{a,\Xi}(r\!-\!\tau),\mathcal{L}_{r}^{S,\Xi})\!-\!\sigma(U^{a,\Xi}(r),U^{a,\Xi}(r\!-\!\tau),\mathcal{L}_{r}^{U,\Xi})\big|^{2}\mathrm{d}r\Big)^{\frac{1}{2}}\bigg]\nn\\
&\leq\frac{1}{4}\mathbb{E}\big[\sup_{0\leq r\leq t}|\Gamma^{a,\Xi}(r)|^{q}\big]\nn\\
&~~~~+C\mathbb{E} \bigg[\!\Big(\int_{0}^{t}\big|\sigma(S^{a,\Xi}(r),S^{a,\Xi}(r-\tau),
\mathcal{L}_{r}^{S,\Xi})\!-\!\sigma(U^{a,\Xi}(r),U^{a,\Xi}(r-\tau),\mathcal{L}_{r}^{U,\Xi})\big|^{2}
\mathrm{d}r\Big)^{\frac{q}{2}}\!\bigg]\nn\\
&\leq\frac{1}{4}\mathbb{E}\big[\sup_{0\leq r\leq t}|\Gamma^{a,\Xi}(r)|^{q}\big]\nn\\
&~~~~+\!C\mathbb{E}\bigg[\!\int_{0}^{t}\!\Big(\big|\!S^{a,\Xi}(r)\!-\!U^{a,\Xi}(r)\!\big|^{2}\!+\!\big|\!S^{a,\Xi}(r\!-\!\tau)\!-\!U^{a,\Xi}(r\!-\!\tau)\big|^{2}
\!+\!\mathbb{W}^{2}_{2}(\mathcal{L}_{r}^{S,\Xi},\!\mathcal{L}_{r}^{U,\Xi})\!\Big)^{\frac{q}{2}}\!\mathrm{d}r\!\bigg] \nn\\
&\leq\frac{1}{4}\mathbb{E}\big[\sup_{0\leq r\leq t}|\Gamma^{a,\Xi}(r)|^{q}\big]
\!+\!C\mathbb{E}\bigg[\int^{t}_{0}\Big(\big|S^{a,\Xi}(r)\!-\!\bar{U}^{a,\Xi}(r)\big|^{q}\!+\!\big|S^{a,\Xi}(r-\tau)\!-\!\bar{U}^{a,\Xi}(r-\tau)\big|^{q}\nn\\
&~~~~+\big|\bar{U}^{a,\Xi}(r)-U^{a,\Xi}(r)\big|^{q}+\big|\bar{U}^{a,\Xi}(r-\tau)-U^{a,\Xi}(r-\tau)\big|^{q}
+\mathbb{W}^{q}_{q}(\mathcal{L}_{r}^{S,\Xi},\mathcal{L}_{r}^{U,\Xi})\Big)\mathrm{d}r\bigg]\nn\\
&\leq\frac{1}{4}\mathbb{E}\big[\sup_{0\leq r\leq t}|\Gamma^{a,\Xi}(r)|^{q}\big]\!+\!C \int^{t}_{0}\sup_{1\leq a\leq \Xi}\mathbb{E}\Big[\sup_{0\leq v\leq r}\big|S^{a,\Xi}(v)\!-\!\bar{U}^{a,\Xi}(v)\big|^{q}\Big]\mathrm{d}r+C\triangle^{\frac{q}{2}} .
\end{align}
Inserting  \eqref{eq4.23}-\eqref{eq4.27} into  \eqref{eq4.22} yields
\begin{align}\label{eq4.28}                                                                                                                                                    &\mathbb{E}\big[\sup_{0\leq r\leq t}|\Gamma^{a,\Xi}(r)|^{q}\big]\leq C \int^{t}_{0}\sup_{1\leq a\leq \Xi}\mathbb{E}\Big[\sup_{0\leq v\leq r}|S^{a,\Xi}(v)-\bar{U}^{a,\Xi}(v)|^{q}\Big]\mathrm{d}r+C\triangle^{\alpha q} .
\end{align}
Hence, thanks to~$(\ref{eq4.21})$~and the Gronwall inequality, we get the desired assertion
$$
\sup_{1\leq a\leq \Xi}\mathbb{E}\Big[\sup_{0\leq t\leq T}\big|S^{a,\Xi}(t)-\bar{U}^{a,\Xi}(t)\big|^{q}\Big]\leq C\triangle^{\alpha q}.
$$
\end{proof}

Finally,   the convergence rate between the numerical solution~$U^{a,\Xi}(t)$~and the exact solution of~$(\ref{eq3.33})$~ follows from Lemma~$\ref{le4.3}$~and Lemma~$\ref{le4.4}$~immediately.
\begin{theorem}\label{th4.5}
Under Assumptions \ref{a1} and \ref{a2} with~$p\geq4(c+1)$,  for any~$q\in[2,\frac{p}{2(c+1)}]$,
$$
\sup_{1\leq a\leq \Xi}\sup_{0\leq t\leq T}\mathbb{E}\big[|S^{a,\Xi}(t)-U^{a,\Xi}(t)|^{q}\big]\leq C\triangle^{\alpha q},~T\geq0. $$
\end{theorem}

In the above theorem choosing $\alpha= {1}/{2}$, we can get the optimal convergence rate $1/2$. Therefore,  the desired rate of convergence follows directly from Lemma~$\ref{le3.6}$.
\begin{theorem}\label{th4.6}
Under Assumptions \ref{a1} and \ref{a2} with~$p\geq4(c+1)$,   for any~$T\geq0$,~$\Xi\geq2$,
\begin{align*}
\displaystyle\sup_{0\leq t\leq T}\mathbb{E}\Big[\big|S^{a}(t)-U^{a,\Xi}(t)\big|^{2}\Big]\leq C\left\{
\begin{array}{lll}
\Xi^{-1/2}+\triangle^{2\alpha},~~~&1\leq d<4,\\
\Xi^{-1/2}\log(\Xi)+\triangle^{2\alpha},~~~&d=4,\\
\Xi^{-d/2}+\triangle^{2\alpha},~~~&4<d.
\end{array}
\right.
\end{align*}

\end{theorem}

\section{ Numerical Example}
Consider the scalar MV-NSDDE (refer to \cite[pp.609, Remark 2.3]{TC2018})
\begin{align}\label{eq5.1}
\mathrm{d}\big(S(t)+\beta S(t-\tau)\big)&=\Big(S(t)-S^{3}(t)+\beta S(t-1)-\beta^{3}S^{3}(t-1)+\mathbb{E}[S(t)]\Big)\mathrm{d}t\nn\\
&~~~~+\big(S(t)+\beta S(t-\tau)\big)\mathrm{d}B(t),~~~~~~~~~~~~~t\geq0.
\end{align}
The initial data~is~$S(t)=\xi(t)=t,$ $ t\in[-\tau,0]$. Let~$\tau= {1}/{32}$,~$T=1$~and~$\beta= {1}/{2}$.
\begin{figure}[!htpb]
  \centering
\includegraphics[width=12cm,height=7.2cm]{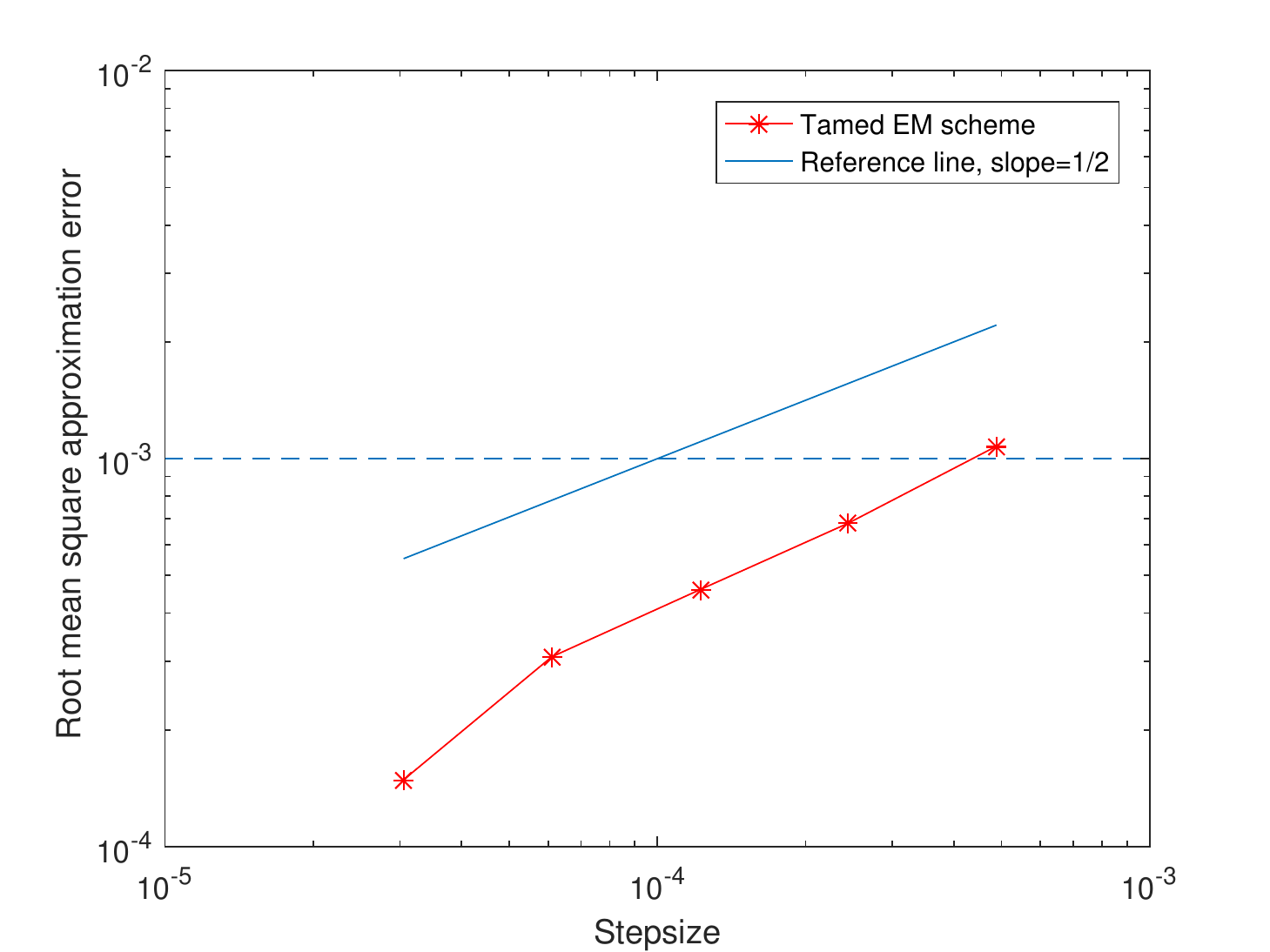}
  \caption{The red one is the root sample mean square
  error~$\big(\mathbb{E}[|S(T)-U(T)|^{2}]\big)^{\frac{1}{2}}$~between the exact solution~$S(T)$ and the numerical solution~$U(T)$ with $\Xi=1000$, as a function of~$\triangle\in\{2^{-15},2^{-14},2^{-13},2^{-12},2^{-11}\}$, while the blue one is the reference line with slope $ {1}/{2}$.}
\label{figure1}
\end{figure}
By computation, one observes that~$(\ref{eq3.2})$~and Assumptions~$\ref{a1}$,~$\ref{a2}$~hold with~all $p\geq2$~and~$c=2$. Thus by virtue of Theorem \ref{th3.5} it follows that $(\ref{eq5.1})$~has a unique solution with this initial data.

Furthermore, let $\alpha= {1}/{2}$, and  fix $\triangle\in(0,1/32)$ such that $n_{0}= {1}/{(32\triangle)}$ is an integer. By Theorem~$\ref{th4.6}$, the approximation solution of  the tamed scheme~$(\ref{eq4.2})$ converges to the exact solution of $(\ref{eq5.1})$  with error estimate in stepsize~$\triangle$ under the sense of mean square. For the numerical experiments we take  the numerical solution with the small size $\triangle=2^{-16}$~as the exact solution $S(\cdot)$  of $(\ref{eq5.1})$.

Let $\Xi=1000$. The Figure~$\ref{figure1}$~shows the root mean square approximation error~$\big(\mathbb{E}[|S(T)-U(T)|^{2}]\big)^{\frac{1}{2}}$~between the exact solution~$S(T)$ of $(\ref{eq5.1})$ and the numerical solution~$U(T)$~of the tamed EM scheme~$(\ref{eq4.2})$, as a function of stepsize~$\triangle\in\{2^{-15},2^{-14},2^{-13},2^{-12},2^{-11}\}$. We observe that this numerical solution~$U(T)$~performs well to approximate the exact solution to~$(\ref{eq5.1})$~with almost  ~$1/2$ order  convergence rate.
\qed

\section*{ Acknowledgements}
The research of the first author was  supported in part by the National Natural Science Foundation of China (11971096), the Natural Science Foundation of Jilin Province (YDZJ202101ZYTS154), the Education Department of Jilin Province (JJKH20211272KJ), the Fundamental Research Funds for the
Central Universities.

\end{document}